\documentclass[11pt,reqno]{amsart}

\usepackage{amsmath,amsfonts,amsthm,amssymb}
\usepackage{color,soul}
\usepackage{graphicx}
\usepackage{verbatim}
\usepackage{color}
\usepackage{amscd}
\usepackage{verbatim}
\usepackage{mathrsfs}
\usepackage[colorlinks,citecolor=red,pagebackref,hypertexnames=false, breaklinks]{hyperref}
\usepackage{esint}
\usepackage{stmaryrd}
\usepackage{enumitem}

\begin{document}
\newcommand{\M}{{\mathcal M}}
\newcommand{\loc}{{\mathrm{loc}}}
\newcommand{\core}{C_0^{\infty}(\Omega)}
\newcommand{\sob}{W^{1,p}(\Omega)}
\newcommand{\sobloc}{W^{1,p}_{\mathrm{loc}}(\Omega)}
\newcommand{\merhav}{{\mathcal D}^{1,p}}
\newcommand{\be}{\begin{equation}}
\newcommand{\ee}{\end{equation}}
\newcommand{\mysection}[1]{\section{#1}\setcounter{equation}{0}}
\newcommand{\laplace}{\Delta}
\newcommand{\pl}{\laplace_p}
\newcommand{\grad}{\nabla}
\newcommand{\pd}{\partial}
\newcommand{\bo}{\pd}
\newcommand{\csub}{\subset \subset}
\newcommand{\sm}{\setminus}
\newcommand{\ssm}{:}
\newcommand{\diver}{\mathrm{div}\,}
\newcommand{\bea}{\begin{eqnarray}}
\newcommand{\eea}{\end{eqnarray}}
\newcommand{\bean}{\begin{eqnarray*}}
\newcommand{\eean}{\end{eqnarray*}}
\newcommand{\thkl}{\rule[-.5mm]{.3mm}{3mm}}
\newcommand{\cw}{\stackrel{\rightharpoonup}{\rightharpoonup}}
\newcommand{\id}{\operatorname{id}}
\newcommand{\supp}{\operatorname{supp}}
\newcommand{\calE}{\mathcal{E}}
\newcommand{\calF}{\mathcal{F}}
\newcommand{\wlim}{\mbox{ w-lim }}
\newcommand{\mymu}{{x_N^{-p_*}}}
\newcommand{\R}{{\mathbb R}}
\newcommand{\N}{{\mathbb N}}
\newcommand{\Z}{{\mathbb Z}}
\newcommand{\Q}{{\mathbb Q}}
\newcommand{\abs}[1]{\lvert#1\rvert}
\newtheorem{theorem}{Theorem}[section]
\newtheorem{corollary}[theorem]{Corollary}
\newtheorem{lemma}[theorem]{Lemma}
\newtheorem{notation}[theorem]{Notation}
\newtheorem{definition}[theorem]{Definition}
\newtheorem{remark}[theorem]{Remark}
\newtheorem{proposition}[theorem]{Proposition}
\newtheorem{assertion}[theorem]{Assertion}
\newtheorem{problem}[theorem]{Problem}
\newtheorem{conjecture}[theorem]{Conjecture}
\newtheorem{question}[theorem]{Question}
\newtheorem{example}[theorem]{Example}
\newtheorem{Thm}[theorem]{Theorem}
\newtheorem{Lem}[theorem]{Lemma}
\newtheorem{Pro}[theorem]{Proposition}
\newtheorem{Def}[theorem]{Definition}
\newtheorem{Exa}[theorem]{Example}
\newtheorem{Exs}[theorem]{Examples}
\newtheorem{Rems}[theorem]{Remarks}
\newtheorem{Rem}[theorem]{Remark}

\newtheorem{Cor}[theorem]{Corollary}
\newtheorem{Conj}[theorem]{Conjecture}
\newtheorem{Prob}[theorem]{Problem}
\newtheorem{Ques}[theorem]{Question}
\newtheorem*{corollary*}{Corollary}
\newtheorem*{theorem*}{Theorem}
\newcommand{\pf}{\noindent \mbox{{\bf Proof}: }}


\renewcommand{\theequation}{\thesection.\arabic{equation}}
\catcode`@=11 \@addtoreset{equation}{section} \catcode`@=12
\newcommand{\Real}{\mathbb{R}}
\newcommand{\real}{\mathbb{R}}
\newcommand{\Nat}{\mathbb{N}}
\newcommand{\ZZ}{\mathbb{Z}}
\newcommand{\CC}{\mathbb{C}}
\newcommand{\Pess}{\opname{Pess}}
\newcommand{\Proof}{\mbox{\noindent {\bf Proof} \hspace{2mm}}}
\newcommand{\mbinom}[2]{\left (\!\!{\renewcommand{\arraystretch}{0.5}
\mbox{$\begin{array}[c]{c}  #1\\ #2  \end{array}$}}\!\! \right )}
\newcommand{\brang}[1]{\langle #1 \rangle}
\newcommand{\vstrut}[1]{\rule{0mm}{#1mm}}
\newcommand{\rec}[1]{\frac{1}{#1}}
\newcommand{\set}[1]{\{#1\}}
\newcommand{\dist}[2]{$\mbox{\rm dist}\,(#1,#2)$}
\newcommand{\opname}[1]{\mbox{\rm #1}\,}
\newcommand{\mb}[1]{\;\mbox{ #1 }\;}
\newcommand{\undersym}[2]
 {{\renewcommand{\arraystretch}{0.5}  \mbox{$\begin{array}[t]{c}
 #1\\ #2  \end{array}$}}}
\newlength{\wex}  \newlength{\hex}
\newcommand{\understack}[3]{%
 \settowidth{\wex}{\mbox{$#3$}} \settoheight{\hex}{\mbox{$#1$}}
 \hspace{\wex}  \raisebox{-1.2\hex}{\makebox[-\wex][c]{$#2$}}
 \makebox[\wex][c]{$#1$}   }%
\newcommand{\smit}[1]{\mbox{\small \it #1}}
\newcommand{\lgit}[1]{\mbox{\large \it #1}}
\newcommand{\scts}[1]{\scriptstyle #1}
\newcommand{\scss}[1]{\scriptscriptstyle #1}
\newcommand{\txts}[1]{\textstyle #1}
\newcommand{\dsps}[1]{\displaystyle #1}
\newcommand{\dx}{\,\mathrm{d}x}
\newcommand{\dy}{\,\mathrm{d}y}
\newcommand{\dz}{\,\mathrm{d}z}
\newcommand{\dt}{\,\mathrm{d}t}
\newcommand{\dr}{\,\mathrm{d}r}
\newcommand{\du}{\,\mathrm{d}u}
\newcommand{\dv}{\,\mathrm{d}v}
\newcommand{\dV}{\,\mathrm{d}V}
\newcommand{\ds}{\,\mathrm{d}s}
\newcommand{\dS}{\,\mathrm{d}S}
\newcommand{\dk}{\,\mathrm{d}k}

\newcommand{\dphi}{\,\mathrm{d}\phi}
\newcommand{\dtau}{\,\mathrm{d}\tau}
\newcommand{\dxi}{\,\mathrm{d}\xi}
\newcommand{\deta}{\,\mathrm{d}\eta}
\newcommand{\dsigma}{\,\mathrm{d}\sigma}
\newcommand{\dtheta}{\,\mathrm{d}\theta}
\newcommand{\dnu}{\,\mathrm{d}\nu}
\newcommand{\Ker}{\mathrm{Ker}}
\newcommand{\Ima}{\mathrm{Im}}

\def\Xint#1{\mathchoice
{\XXint\displaystyle\textstyle{#1}}%
{\XXint\textstyle\scriptstyle{#1}}%
{\XXint\scriptstyle\scriptscriptstyle{#1}}%
{\XXint\scriptscriptstyle\scriptscriptstyle{#1}}%
\!\int}
\def\XXint#1#2#3{{\setbox0=\hbox{$#1{#2#3}{\int}$ }
\vcenter{\hbox{$#2#3$ }}\kern-.6\wd0}}
\def\dashint{\Xint-}

\newcommand{\Rd}{\color{red}}
\newcommand{\Bk}{\color{black}}
\newcommand{\Mg}{\color{magenta}}
\newcommand{\Wh}{\color{white}}
\newcommand{\Bl}{\color{blue}}
\newcommand{\Yl}{\color{yellow}}


\renewcommand{\div}{\mathrm{div}}
\newcommand{\red}[1]{{\color{red} #1}}

\newcommand{\cqfd}{\begin{flushright}                  
			 $\Box$
                 \end{flushright}}
                 
\newcommand{\todo}[1]{\vspace{5 mm}\par \noindent
\marginpar{\textsc{}} \framebox{\begin{minipage}[c]{0.95
\textwidth} \tt #1
\end{minipage}}\vspace{5 mm}\par}


\title[] {Infinite graphs satisfying the Bakry-Emery curvature condition
 $CD(0,n)$: The modified heat equation and applications to geometric analysis}

\author{Herv\'e Pajot and Emmanuel Russ}
\address{Herv\'e Pajot, Institut Fourier - Universit\'e Grenoble Alpes, France}
\email{herve.pajot@univ-grenoble-alpes.fr}
\address{Emmanuel Russ, Institut de Math\'ematiques de Marseille - Aix-Marseille Universit\'e, France}
\email{emmanuel.russ@univ-amu.fr}

\maketitle

\begin{abstract}
Let $G=(V,p, \mu)$ be a (finite or infinite) weighted graph with bounded geometry. Assuming that $G$ satisfies the classical curvature-dimension condition of Bakry-Emery $CD(K,n)$  with $K\ge 0$ (for the usual Laplacian), we prove that the doubling volume property holds. One of the key points is to establish the existence and uniqueness of solutions of a modified non linear heat equation which replaces the standard one usually used in the case of Riemannian manifolds.  Li-Yau and Harnack estimates for the solutions of this modified heat equation are obtained. We also provide explicit examples of Cayley graphs satisfying our assumptions.
\end{abstract}

\tableofcontents

\section{Introduction}

 Let \((M,g)\) be a complete Riemannian manifold  of dimension $n$ with nonnegative Ricci curvature. It is a well-known fact that there exists $C_{n}>0$ such that, for all \(x\in M\) and all \(r>0\),
\[
Vol_{g}(B(x,2r))\le C_{n} Vol_{g}(B(x,r)),
\]
where $B(x,r)$ is the geodesic open ball with center $x$ and radius $r$ and $Vol_{g}$ denotes the Riemannian measure. This fact, known as the doubling volume property, can be derived from the Bishop-Gromov comparison Theorem which gives $C_{n}=2^{n}$  (\cite[Theorem 4.19]{GHL}, \cite[Theorem III.4.5]{Chavel}) and plays a key role in geometric analysis (see for instance \cite{PR}, chapter 4, or \cite{SC}). \Bk\par
\noindent Recently, several attemps to extend the notion of Ricci curvature bounds,  initially defined in Riemannian manifolds, to the case of general metric measured spaces, \Bk were proposed (see for instance \cite{LottVillani, Villani2, Sturm1, Sturm2}) by using optimal transportation. In the discrete setting, several possible definitions were proposed. An approach using mass transportation  was initiated by Ollivier \cite{Ollivier}. An interesting point of view is to consider the curvature-dimension condition introduced by Bakry-Emery \cite{BE} in the case of diffusions on manifolds, related to an elliptic operator as the Laplace-Beltrami operator (see \cite{Ledoux} or \cite{Logsob}) . It turns out that this condition, called $CD(K,n)$ in the case of a Riemannian manifold of dimension $n$  (equipped with the Laplace-Beltrami operator)  is equivalent to the fact that the Ricci curvature is bounded below by $K$ (see for instance \cite[Theorem 14.8]{Villani2}).\par
\noindent One of the main difficulties when proving geometric consequences of the Bakry-Emery condition in a discrete setting, such as the doubling volume property, lies in the lack of chain rule. To cope with this problem, several recent works considered modified versions of the Bakry-Emery criterion (see for instance \cite{BHLLMY,HLLY} and the remark \ref{rk:CDE} below) and are related to the classical heat equation.  These modified versions are stronger than the classical one and do not seem to have geometric interpretation in the Riemannian case. \Bk In \cite{M}, in the case of finite graphs, it was proved that the classical Bakry-Emery $CD(K,n)$ condition implies the doubling volume \Bk property, with a constant only involving $n$ and an ellipticity constant of the Laplacian. The proof relies on properties of the solutions of a modified nonlinear heat equation. The main result in the present paper states that the doubling volume property still holds for (possibily infinite) graphs with bounded geometry satisfying $CD(K,n)$ where $K\ge 0$ and $n$ is a nonnegative integer. The proofs follow a general strategy analogous to the one in \cite{M}, but several difficulties arise, due to the fact that the graph is not assumed to be finite.  \par

\medskip

\noindent Before stating our results, we  introduce a couple of natural and geometric conditions on the graph $G$  (all the precise definitions will be given and discussed in Section \ref{sec:setting}).  A weighted graph $G$ is defined as a triple $G=(V,p,\mu)$ where $V$ is a countable set of points called vertices, $p : V \times V \rightarrow [0, \infty)$ is a Markov kernel and $\mu$ is a measure on all subsets of $V$, that we first define on vertices and then extend by additivity to subsets of $V$. For all $x,y\in V$, say that $x$ and $y$ are neighbours (which will be denoted by $x\sim y$) if and only if $p(x,y)>0$. The kernel $p$ is not assumed to be symmetric, but we suppose that the Markov chain associated with $p$ is reversible with respect to $\mu$. This implies that, for all $x,y\in V$, $x\sim y$ if and only if 
$y\sim x$.\par 
\noindent We will define on $G$ a natural Laplace operator $\Delta$ given, for all functions $u$ on $V$ and all $x\in V$ by
\[
\Delta u(x):= \sum_{y\sim x} p(x,y)(u(y)-u(x)),
\]  so that $u : V \rightarrow \R$ is harmonic (that is $\Delta u=0$) if and only if $u$ satisfies  a suitable version of the \Bk mean value property (as in the classical Euclidean situation). \par\Bk  
\noindent We will consider in this paper graphs with bounded geometry, that is graphs $G=(V,p,\mu)$ so that the Markov kernel $p$ is reversible with respect to the measure $\mu$ and the Laplace operator $\Delta$ satisfies an ellipticity condition.
\medskip

\noindent Our main result states that, assuming that $G$ satisfies the $CD(0,n)$ condition for some integer $n\ge 1$ (see \eqref{eq:CDKn} below), the graph $G$ satisfies the doubling volume property:
\begin{theorem} \label{th:main}
Let $G$ be a weighted graph with bounded geometry. Assume also that \eqref{eq:CDKn} holds for $G$ with $K=0$ and some $n\in \N^{\ast}$. Then there exists $C_{DV}>0$ depending on $n$ and on the constant $\alpha$ given by an ellipticity condition of the Laplacian $\Delta$ on $G$  such that, for all $x\in V$ and all $r>0$,
\begin{equation} \label{eq:DV} \tag{$DV$}
\mu(B(x,2r))\le C_{DV}\mu(B(x,r)).
\end{equation}
\end{theorem}
\noindent Note that the constant $\alpha$ will play a key role in many  of our estimates.
The condition $CD(K,n)$, which will be precisely defined in Section \ref{sec:BK}, means that the curvature is bounded below by $K$, whereas the parameter $n$ has in general no geometric interpretation (except in the Riemannian case for instance).

\noindent Since, due to the discrete setting, the chain rule on $G$ does not hold, one cannot directly use qualitative properties of solutions of the heat equation to derive the doubling property from the Bakry-Emery condition. \Bk Following ideas of \cite{M}, the proof of Theorem \ref{th:main} relies on a parabolic Harnack inequality for solutions of a nonlinear heat equation on $V$ with quadratic growth of the gradient. In the Euclidean case, the modified heat equation could be written as 
\begin{equation}\label{eq:heatmodif0}
\frac{\partial u}{\partial t}(x,t)=\Delta_{x} u(x,t)+
|\nabla_{x} u|^{2}(x,t).
\end{equation}
This PDE is a parabolic semilinear equation, called viscous Hamilton-Jacobi equation in \cite{AB}. 

\begin{remark} \label{rk:heat-log}
If $u$ is a solution of \eqref{eq:heatmodif0} in the Euclidean setting, then (\cite[Section 2]{GGK}) $w:=e^{u}$ solves the linear heat equation
\[
\frac{\partial w}{\partial t}=\Delta w.
\]
Thus, solutions of \eqref{eq:heatmodif0} can be seen as substitutes of $\log w$ when $w$ solves the linear heat equation. \Bk\par
\end{remark}
\noindent In our discrete setting, given 
$T\in (0,\infty]$ and $p\in [1,\infty]$, \Bk a function $u:[0,T]\rightarrow L^p(V,\mu)$ is called a solution on some interval $[0,T]$ of the modified heat equation if $u$ is continuous on $[0,T]$, $C^1$ in $(0,T]$ and is a solution of
\begin{equation}\label{eq:heatmodif1}
\frac{\partial u_{t}}{\partial t}=\Delta u_t+\Gamma u_t,
\end{equation}
where, for all functions $v$ on $V$,
\[
\Gamma v(x):=\frac{1}{2} 
    \sum_{y \sim x} p(x,y) (v(x)-v(y))^{2}.
\]
Note that $\Gamma v$ can be considered as the square of the length of the gradient of $v$ and is usually called "carr\'e du champ" operator. Inspired by arguments in \cite{AB} for the existence and the uniqueness of solutions of the heat equation with nonlinearities involving the gradient term in the Euclidean case, \Bk we prove existence and uniqueness of the solutions of (\ref{eq:heatmodif1}). A crucial point will be to get exponential decay for the gradients of solutions of \eqref{eq:heatmodif1}.
\begin{theorem} \label{thm:solmodif}
Let $G$ be a weighted graph with bounded geometry and satisfying \eqref{eq:CDKinfty} with $K \geq 0$.
Let $u_0\in L^\infty(V,\mu)$. If 
\[
\left\Vert \Gamma u_0\right\Vert_\infty<\alpha/2
\]
where $\alpha$ is given by the ellipticity condition for the Laplacian on $G$,
then there exists a unique solution $u:[0,\infty)\rightarrow L^\infty(V,\mu)$ of \eqref{eq:heatmodif1} such that $u(0)=u_0$. Moreover, for all $t\ge 0$, 
\begin{equation} \label{eq:estimgammaut}
\left\Vert \Gamma u(t)\right\Vert_\infty\le e^{-2Kt}\left\Vert \Gamma u_0\right\Vert_\infty.
\end{equation}
\end{theorem}
\noindent Let us now sketch the proof of Theorem \ref{thm:solmodif}.  First, for any $p \in [1,\infty]$, the Laplacian on a weighted graph with bounded geometry  generates a uniformly continuous semigroup on $L^{p}(V,\mu)$ denoted by $(P_{t})_{t >0}$. Then, we prove the existence of solution of the modified heat equation for small time. To do this, for $p \in [1,\infty]$ and $T>0$,  
the standard theory of semigroups (\cite{P}) yields that   if $f\in L^1(0,T;L^p(V,\mu))$ is  a continuous function on $(0,T]$, 
then, for all $u_0\in L^p(V,\mu)$, the  inhomogeneous Cauchy problem 
\begin{equation} 
\left\{
\begin{array}{ll}
\frac{du}{dt}=\Delta u(t)+f(t) & \mbox{ for all }t\in [0,T),\\
u(0)=u_0.
\end{array}
\right.
\end{equation}
 has a unique solution $u$, continuous on $[0,T)$, $C^1$ on $(0,T)$ and given by
\[
u(t)=P_tu_0+\int_0^t P_{t-s}f(s)ds.
\]
In the next step, under the assumption \eqref{eq:CDKinfty}, we solve the Cauchy problem for the ``modified'' heat equation.  Namely, given $u_0\in L^\infty(V,\mu)$ non constant, we prove that there exists a unique solution $u:[0,T)\rightarrow L^\infty(V,\mu)$ of:
\[
\left\{
\begin{array}{ll}
\frac{du}{dt}=\Delta u(t)+\Gamma u(t) & \mbox{ for all }t\in (0,T),\\
u(0)=u_0,
\end{array}
\right.
\]
where $T=\frac 1{256}\left\Vert \Gamma u_0\right\Vert_\infty^{-1}$. 
To that purpose, following ideas from \cite{AB}, we construct  a sequence of functions $(u^k)_{k\ge -1}:[0,T)\rightarrow L^\infty(V,\mu)$, continuous on $[0,T)$ and $C^1$ on $(0,T)$, such that $u^{-1}=0$, solving, for all $k\in\N$, 
\begin{equation}
\left\{
\begin{array}{ll}
\frac{du^k}{dt}=\Delta u^k+\Gamma u^{k-1} & \mbox{ for all }t\in (0,T),\\
u^k(0)=u_0,
\end{array}
\right.
\end{equation}
and satisfying for all $k \geq -1$ 
$$M_{k}:=\sup_{0\le t<T} \left\Vert \sqrt{\Gamma u^{k}(t)}\right\Vert_\infty\le 2\left\Vert \sqrt{\Gamma u_0}\right\Vert_\infty.$$
The solution $u_{t}$ is obtained as the uniform limit of the sequence $(u^{k})_{k\ge -1}$.\par
\medskip
\noindent Assuming furthermore that $\left\Vert \Gamma u_0\right\Vert_\infty<\frac{\alpha}2$, we prove that
\[
\left\Vert \Gamma u(t)\right\Vert_\infty\le \left\Vert \Gamma u_0\right\Vert_\infty
\]
for all $t\in [0,T)$. Using this estimate on $[0,\frac T2]$ and iterating, one obtains the conclusion of Theorem \ref{thm:solmodif}.\par

\medskip

\noindent On a Riemannian manifold $M$ of dimension $n$  with nonnegative Ricci curvature, the Li-Yau inequality (\cite{LY}) states that 
$$ - \Delta_{M} \log P^{M}_{t}f \leq \frac{n}{2t}$$
for all positive functions   $f$ on $M$ and all $t>0$, where $P^{M}_{t}=e^{t\Delta_{M}}$ is the heat semi-group and $\Delta_{M}$ is the Laplace-Beltrami operator on  $M$. A first application of Theorem \ref{thm:solmodif} is a Li-Yau type estimate for solutions of the modified heat equation. Note that, in Theorem \ref{pro:liyau} below, $\log P^{M}_{t}$ is replaced by the solution $u_{t}$ of the modified heat equation, which is consistent with the link between solutions of the modified heat equation and the linear one explained above. Observe also that the conclusion of Theorem \ref{pro:liyau} holds under assumption $CD(0,n)$, while \eqref{eq:CDKinfty} is enough for the conclusion of Theorem \ref{thm:solmodif} to be true. 
\begin{theorem} \label{pro:liyau}
  Let $G$ be a weighted graph with bounded geometry that satisfies $CD(0,n)$ for some $n < \infty$. Let $u_t$ be a solution of \eqref{eq:heatmodif1}.
  Then, for all $t >0$,
  $$ \Gamma u_{t} - \partial_{t} u_{t} = - \Delta u_{t} \leq \frac{n}{2t}.$$
\end{theorem}
\noindent As an application of Theorem \ref{pro:liyau}, we derive a Harnack type inequality for solutions of \eqref{eq:heatmodif1} (see also \cite{CLY} for Harnack inequalities for eigenfunctions of the Laplacian in the case of finite graphs). 
\begin{theorem} \label{pro:Harnack}
  Let $G$ be a weighted graph with bounded geometry that satisfies $CD(0,n)$ for some $n < \infty$. Let $u$ be a solution of \eqref{eq:heatmodif1}. Then, 
  for all $0< T_{1} < T_{2}$ and all $x,y \in V$,
 $$ u_{T_{1}}(x) -u_{T_{2}}(y)\leq\frac{n}{2} \log \left(  \frac{T_{2}}{T_{1}} \right) +2 \frac{d(x,y)^{2}}{\alpha (T_{2}-T_{1})}$$
  where $\alpha$ is as in Theorem \ref{thm:solmodif}.
\end{theorem}
\noindent To establish \eqref{eq:DV}, we apply all our previous results, in particular the Harnack inequality, to the solution of \eqref{eq:heatmodif1} with a suitable initial data 
$u_{0} \in L^{\infty}(V,\mu)$. In the Riemannian case, a natural choice is to consider $u_{0}= \chi_{B(x,r)}$ the characteristic function of the ball $B(x,r)$. In our setting, we consider an approximation of $\log \chi_{B(x,r)}$ (which is not well-defined !).\\

The paper is organized as follows. In Section \ref{sec:setting}, we introduce the basic definitions and properties. In particular, we prove some estimates in the setting of $\Gamma$-calculus in the sense of Bakry-Emery. We focus on the case of Cayley graphs of Abelian groups in Section \ref{sec:abelian}, proving that condition \eqref{eq:CDKn} holds in different contexts. The heat semigroup is presented in Section \ref{sec:heat}. Section \ref{sec:modifiedheat} is devoted to the proof of Theorem \ref{thm:solmodif}. Some pointwise upper and lower bounds for exponentials of the solutions of the modified heat equation are proved in Section \ref{sec:compar}, while the Li-Yau and Harnack inequalities for solutions of the modified heat equation are established in Section \ref{sec:liyau}. Finally, the doubling volume property is derived in Section \ref{sec:doubling}.

\section{The general setting} \label{sec:setting}
\subsection{Graphs}
A graph $G$ is given by  a triple $(V,p,\mu)$ that we now describe.\\
Let $V$ be a countable (finite or infinite) set. Consider a map $p : V \times V \rightarrow [0, + \infty$). 
For any $x \in V$, say that $y$ is a neighbor of $x$
which will be denoted by $x\sim y$, if and only if
$p(x,y)>0$. Note that $p$ is not assumed to be symmetric (which would mean that $p(x,y)=p(y,x)$ for all $x$, $y \in V$),  but, we will see later that $x \sim y \Longleftrightarrow y \sim x$ under our assumptions.\par
\noindent We  will always assume that $(V,p)$ is locally finite, in the sense that for any $x \in V$, the set of neighbors of $x$ is finite.  The cardinal of this set is called the valence of $x$
 and is denoted by $N(x)$. 
Moreover, we assume that $p$ is a Markov kernel, that is  for any $x \in V$, 
\begin{equation} \label{eq:Markov}
\sum_{y\in V;\ x\sim y} p(x,y)=1.
\end{equation}
Thus, $p$ defines a random walk on $V$.\par
\noindent Consider a map $\mu : V \rightarrow (0, \infty)$. We can extend $\mu$ to a measure on all subsets of $V$ by setting $\mu(A)=\sum_{x \in A} \mu(x)$ 
whenever $A \subset V$. \par
\noindent Let us now define a distance on a locally finite graph.
\begin{Def} \label{def:dist}
Let $G=(V,p,\mu)$ be a locally finite graph.
\begin{enumerate}
\item For $x,y\in V$, a path joining $x$ to $y$ is a finite sequence of vertices $x_0=x,...,x_N=y$ such that $N\ge 1$ is an integer and, for all $0\leq i\leq N-1$, $x_i\sim x_{i+1}$. By definition, the length of such a path is $N$. 
\item Say that $G$ is connected if and only if, for all $x,y\in V$, there exists a path joining 
$x$ to $y$. 
\item For all $x,y\in V$, the distance between $x$ and $y$, 
denoted by $d(x,y)$, is defined as $d(x,y)=0$ if $x=y$ and $d(x,y)$ is the shortest length of a path joining $x$ and $y$ if $x\ne y$.
\item For all $x\in V$ and all $r\geq 0$, let 
$B(x,r):=\left\{y\in V,\ d(y,x)\leq r\right\}$. For all $x\in V$ and all $r\ge 0$, set $V(x,r):=\mu(B(x,r))= \sum_{y \in B(x,r)} \mu(y)$.\\
\end{enumerate}
\end{Def}

\begin{Def}
A weighted graph $G$ is a triple $G=(V, p, \mu)$ where $V$ is a countable set, $p$ is Markov kernel on $V \times V$ and $\mu$ is a measure on $V$ meeting the two following conditions:
\begin{enumerate}
\item $G$ is locally finite (recall that this means that, for all $x\in V$, the set $\left\{y\in V;\ x\sim y\right\}$ is finite),
\item $G$ is connected.
\end{enumerate}
\end{Def}
\begin{example}[Random conductance model]\label{ex:random}
Let $V$ be a countable set equipped with an unoriented collection of edges ${\mathcal E}$ (which means that
$\mathcal{E}$ is a set of unordered pairs $\left\{x,y\right\}$  of vertices). We say that $x \sim y$ if $\left\{x,y\right\} \in \mathcal{E}$. Note that, for all $x,y\in V$, $x\sim y$ if and only if $y\sim x$. Assume that, for all $x\in V$, the set $\left\{y\in V;\ x\sim y\right\}$ is non-empty and finite. Suppose that we are given a collection of numbers $(\omega_{xy})_{(x,y) \in V \times V}$ so that $\omega_{x,y}\geq 0$ and $\omega_{x,y}>0$ if and only if $x\sim y$. If we set for all $x\in V$,
\begin{equation} \label{eq:m(x)}
m(x):=\sum_{y \in V} \omega_{xy}= \sum_{y \sim x} \omega_{xy},
\end{equation}
then, if $m(x) \in (0,+ \infty)$. For $x$, $y \in V$, define $p(x,y)= \displaystyle \frac{\omega_{xy}}{m(x)}$. Then, $G=(V, p,m)$ is a weighted graph (if we assume that $G$ is connected). The number $\omega_{xy}$ is usually called the conductance of the pair $\left\{x,y\right\}$ or the weight of the edge $xy$ \cite[Section 1.1]{Biskup}).
\end{example}

\begin{example}[Discretization of manifolds]\label{ex:discretisation}
Another construction of weighted graphs can be described as follows (\cite[Section 3]{CSC1}). Let $V$ be a countable set equipped with a collection of edges ${\mathcal E}$ as in Example \ref{ex:random}, and $\mu$ be a positive function on $V$. Assume that $(V,{\mathcal E})$ is locally finite and connected. For all $x,y\in V$, define 
\begin{equation} \label{eq:defp}
p(x,y)= \displaystyle \frac{\mu(y)^{1/2}}{\sum_{z \sim x}\mu(z)^{1/2}}
\end{equation}
if $x \sim y$ and $p(x,y)=0$ otherwise.  It is plain that $p$ is a Markov kernel and that $p(x,y)>0$ if and only if $x\sim y$. \par
\noindent This construction is used for the discretization of a Riemannian manifold. Let $(M,g)$ be a complete, connected, and non compact Riemannian manifold. We will  denote by $d_{g}$ and $Vol_{g}$ the Riemannian distance and volume. 
  Fix a scale $\varepsilon >0$. Consider a maximal $\varepsilon$-separated set $V_{\varepsilon}$ in $M$, that is 
  $V_{\varepsilon}$ is a maximal subset of $M$ so that, for all $x,y\in V_\varepsilon$ such that $x\ne y$, $d(x,y)\ge \varepsilon$. For any $x$, $y \in V_{\varepsilon}$, say that $x \sim y$ if and only if $d_{g}(x,y) \leq 2 \varepsilon$. Moreover, set $\mu(x)= Vol_{g}(B(x,\varepsilon))$. 
Define $p$ by \eqref{eq:defp}. Then, $G=(V_{\varepsilon}, p, \mu)$ is a weighted graph, used in \cite{CSC1}  to prove that some geometric inequalities are preserved under discretization.  See for instance \cite{Kanai} or \cite{CSC1}  for more details on discretization of manifolds and applications to geometry and analysis. 
\end{example}

\begin{definition} \label{def:bdedgeom}
A weighted graph $G=(V,p,\mu)$ has bounded geometry if and only if $G$ satisfies the following conditions:
\begin{enumerate} [label=(A\arabic*)]
    \item \label{item:A1} There exists $\alpha >0$ so that for any $x \in V$, $p(x,y) \geq \alpha$ whenever $y \sim x$. 
    \item \label{item:A2} The Markov kernel $p$ is reversible
     with respect to the measure $\mu$, in the sense that for any $x \sim y$ in $V$, $p(x,y)\mu(x)= p(y,x) \mu(y)$.
\end{enumerate}
\end{definition}

\begin{remark}
Our terminology could appear to be a little bit unusual. The notion of graphs with bounded geometry concerned sometimes graphs with uniformly bounded valences (see for instance \cite{Chavel} section 4.4 and the references therein). It turns out that, by Proposition \ref{pro:controle} below, graphs with bounded geometry in the sense of Definition \ref{def:bdedgeom} have uniformly bounded valences. Hence, for convenience, we use the same terminology. This is also motivated by the fact that graphs with bounded geometry are related to Riemannian manifolds with bounded geometry via discretization, as will appear in Example \ref{ex:discretisation-bis} below.
\end{remark}

\begin{remark}
\begin{enumerate}
\item Condition \ref{item:A1} is a uniform ellipticity condition for the Laplacian (see Section \ref{sec:lap} for the definition of $\Delta$) and is classical in analysis on discrete graphs. For instance, it is involved in \cite{Del} in conjunction with the assumption that $x\sim x$ for all $x\in V$ to give characterizations of a parabolic Harnack inequality and Gaussian estimates for reversible Markov chains by the mean of some geometric inequalites (doubling volume condition, Poincar\'e inequalities on balls). In the present work, it is not assumed that $x\sim x$ for all $x\in V$. 
\item Condition \ref{item:A2} implies: $x \sim y$ if and only if $y \sim x$, since $\mu(x)\ne 0$ and $\mu(y) \neq 0$. It is a classical assumption in probability (see for instance \cite[Section 2.1]{LPnetworks}). Note also that \ref{item:A2} implies that the Laplacian on $G$ is self-adjoint on $L^{2}(V,\mu)$ (see Lemma \ref{lem:propdelta} below).
\end{enumerate}
\end{remark}

\begin{example}[Random conductance model]
Recall example \ref{ex:random}. Assume that for all $x$, $y \in V$, $\omega_{xy}=\omega_{yx}$. Then, the measure $m$ is reversible for the Markov kernel $p$. The uniform ellipticity condition could be written as $\alpha \mu(x) \leq \omega_{xy}$ 
 whenever $x \sim y$. A standard example is to consider $\omega_{xy}=1$ if $\left\{x,y\right\} \in {\mathcal E}$ and $\omega_{xy}=0$ otherwise. In this case $p(x,y)= \displaystyle \frac{1}{N(x)}$
  and $p$ is reversible with respect to tthe measure given by $\mu(x)= N(x)$. See \cite[Section 1.1]{Biskup} for more details. We will apply this type of construction in Section \ref{sec:abelian} in the setting of Cayley graphs of finitely generated groups. 
\end{example}

\begin{example}\label{ex:discretisation-bis} [Discretization of manifolds]
 Recall example \ref{ex:discretisation}. Consider the measure on $V_{\varepsilon}$ given, for all $x \in V_{\varepsilon}$ by 
 $$\Pi(x)= \mu(x)^{1/2} \left( \sum_{z \sim x} \mu(z)^{1/2} \right).$$
 Then, $p$ is reversible with respect to $\Pi$. Let $\rho(M)$ be the injectivity radius of $M$ and assume that $\rho(M)>0$. Then, if $0<\varepsilon \leq \rho(M)/2$ and $M$ has non negative Ricci curvature,  $p$ satisfies \ref{item:A1}. This is due to the the upper bound $Vol_{g}(B(x,r)) \leq C r^{n}$, which follows in turn from Bishop comparison Theorem (see \cite[Theorem III.4.4]{Chavel}). On the other hand, the lower bound $Vol_{g}(x,r)\ge cr^n$ if $0<r<\frac{\rho(M)}2$ is due to Croke \cite[Proposition 14]{Croke}  and does not require  a control of the Ricci curvature. Such a manifold (with lower bounds on the Ricci curvature  and non vanishing injectivity radius) is often called manifold with bounded geometry. 
 Therefore, the discretization $(V_{\varepsilon},p, \Pi)$ of a manifold with bounded geometry  is a weighted graph with bounded geometry. See for instance \cite{Coulhon} for a related discussion.
\end{example}

\begin{proposition} \label{pro:controle}
 Let $G$ be a weighted graph with bounded geometry. 
 Then:
 \begin{enumerate}
     \item \label{item:valencebd}  the valences of the graph $G$ are uniformly bounded: for any $x \in V$, $N(x) 
     \leq 1/\alpha$ (where $N(x)$ is the cardinal of the set of 
  neighbors of $x$).
 \item \label{item:mu} For  $x$, $y \in V$,
    $\alpha \mu(x) \leq \mu(y)  \leq 
    \alpha^{-1} \mu(x)$ whenever $x \sim y$. 
\end{enumerate}
\end{proposition}

  \begin{proof}
For \eqref{item:valencebd}, let $x \in V$. We have
$1= \sum_{y \sim x} p(x,y) \geq N(x) \alpha$. Therefore, $N(x) \leq 1/\alpha$.\par
\noindent For \eqref{item:mu}, note that for all $x$, $y \in V$ with $x \sim y$, we have, by \ref{item:A1} and \ref{item:A2},
$$\alpha \mu(x)\leq p(x,y) \mu(x) = p(y,x) \mu(y) \leq \mu(y).$$
\end{proof}
\noindent From Proposition \ref{pro:controle}, we derive the local doubling volume property for graphs with bounded geometry.
\begin{corollary}\label{pro:DVloc}
Let $G$ be a weighted graph with bounded geometry. Then, $G$ satisfies the local doubling volume property:
for any $r >0$, for any $x \in V$,
$$
V(x,2r) \leq (1+\alpha^{-2})^{r} V(x,r).
$$
\end{corollary}
\begin{proof}
Fix $x \in V$ and $r >0$. Without loss of generality, we can assume that $r \in \N^{*}$. Set $S(x,r)= \{y \in V; d(x,y)=r \}$. 
Then, we have by Proposition \ref{pro:controle},
\begin{eqnarray*}
V(x,r+1) & \leq& V(x,r)+ \sum_{y \in S(x,r)} \sum_{z \sim y} \mu(z) \\
& \le & V(x,r)+ \alpha^{-1}\sum_{y \in S(x,r)} \mu(y)N(y)\\
& \leq & V(x,r)+  \alpha^{-2}\mu(S(x,r)) \\
& \leq& (1+\alpha^{-2}) V(x,r).
\end{eqnarray*}
By iteration, we get $V(x,2r) \leq (1+\alpha^{-2})^{r} V(x,r)$.
\end{proof}
\Bk

\noindent For all $p\in [1,\infty)$, define $L^p(V,\mu)$ as the space of functions $f$ on $V$ such that
\[
\left\Vert f\right\Vert_{L^p(V,\mu)}^p:=\sum_{x\in V} \left\vert f(x)\right\vert^p\mu(x)<\infty,
\]
while $L^\infty(V,\mu)$ is the space of bounded functions on $V$ equipped with the norm
\[
\left\Vert f\right\Vert_{L^\infty(V,\mu)}:=\sup_{x\in V} \left\vert f(x)\right\vert.
\]
\noindent The space $L^2(V,\mu)$ is a Hilbert space for the usual scalar product given by
\[
\langle f,g\rangle:=\sum_{x\in V} f(x)g(x)\mu(x)
\]
for all $f,g\in L^2(V,\mu)$.\par
\subsection{The discrete Laplacian} \label{sec:lap}
Let $G=(V,p,\mu)$ be a weighted graph with bounded geometry. 
We first give the definition of \Bk the Laplace operator. For all functions $f$ on $V$ and all $x\in V$, set
\[
\Delta f(x):= \sum_{y \sim x}p(x,y) (f(y)-f(x)).
\]

\begin{remark}
Denote by $P$ the operator associated usually to the Markov kernel $p$: 
$$Pf(x)= \sum_{y \sim x } p(x,y) f(y).$$
Then, $-\Delta = I -P$ where $I$ is the identity operator.
\end{remark}
\noindent Let us gather properties of $\Delta$ which will be instrumental in the sequel (see also \cite[Theorem 4.1]{Wo}):
\begin{lemma} \label{lem:propdelta}
Let $G$ be a weighted graph with bounded geometry and $p\in [1,\infty]$. Then:
\begin{enumerate}
    \item \label{item:deltabounded} the operator $\Delta$ is $L^p(V,\mu)$-bounded,
    \item \label{item:deltaself} $\Delta$ is self-adjoint in $L^2(V,\mu)$,
    \item \label{item:deltaut} for all $T>0$, if $u:[0,T]\rightarrow L^p(V,\mu)$ is continuous, then so is $t\mapsto \Delta u(t)$.
\end{enumerate}
\end{lemma}
\begin{proof}
Consider first $f\in L^1(V,\mu)$. Then
\begin{eqnarray*}
\sum_{x\in V} \left\vert \Delta f(x)\right\vert \mu(x) & \le & \sum_{x\in V}\sum_{y\sim x} p(x,y) \left\vert f(y)-f(x)  \right\vert \mu(x)\\
& \le & \sum_{x\in V}\left\vert f(x)\right\vert \left( \sum_{y\sim x} p(x,y) \right) \mu(x)+\sum_{y\in V} \left\vert f(y)\right\vert \left(\sum_{x\sim y}p(x,y) \mu (x) \right)\\
& \leq & ||f||_{L^{1}(V,\mu)} + \sum_{y \in V} |f(y)| \left(  \sum_{x \sim y} p(y,x)\mu(y)\right) \\
& =& 2 ||f||_{L^{1}(V,\mu)},
\end{eqnarray*}
where the third line follows from \ref{item:A2}.\par
\noindent Consider now $f\in L^\infty(V,\mu)$. Then, for all $x\in V$,
\[
\left\vert \Delta f(x)\right\vert\le 2\left\Vert f\right\Vert_{L^{\infty}(V,\mu)} \left(  \sum_{y \sim x} p(x,y)\right) = 2 \left\Vert f\right\Vert_{L^{\infty}(V,\mu)}.
\]
By interpolation, for all $p\in [1,\infty]$ and all $f\in L^p(V,\mu)$,
\[
\left\Vert \Delta f\right\Vert_{L^p(V,\mu)}\le 2\left\Vert f\right\Vert_{L^p(V,\mu)},
\]
which proves \eqref{item:deltabounded}.\par
\noindent For \eqref{item:deltaself}, observe that, by \ref{item:A2}, for all functions $f$ and $g$ in $L^2(V,\mu)$,
\begin{eqnarray*}
\sum_{x\in V} g(x)\Delta f(x)\mu(x)  & = &  \sum_{x\in V}\sum_{y\in V}p(x,y)(f(y)-f(x))g(x) \mu(x)\\
 & = &  \sum_{x\in V}\sum_{y\in V}p(y,x) (f(y)-f(x))g(x) \mu(y)\\
& = & \sum_{x\in V}\sum_{y\in V} p(x,y)(f(x)-f(y))g(y) \mu(x),
\end{eqnarray*}
which entails
\begin{equation} \label{eq:IPP}
\sum_{x\in V} g(x)\Delta f(x)\mu(x) =-\frac 12\sum_{x,y\in V} p(x,y) (f(y)-f(x))(g(y)-g(x)) \mu(x).
\end{equation}
Thus, we get \[
\langle f,\Delta g\rangle= \langle \Delta f, g \rangle
\]
for all $f,g\in L^2(V,\mu)$. Since $\Delta$ is $L^2(V,\mu)$-bounded, $\Delta$ is therefore self-adjoint on $L^{2}(V,\mu)$.\par
\noindent Finally, assertion \eqref{item:deltaut} is an immediate consequence of the $L^p$-boundedness of $\Delta$.

 \begin{remark}\label{rem:laplacian}
 Recall example \ref{ex:random}. Consider a measure $\mu$ on all subsets of $V$, such that $\mu(x)>0$ for all $x\in V$.
In this case, another possible definition of the Laplace operator (\cite[Section 2.1]{HLLY}) is given by 
\begin{equation} \label{def:Deltabis}
\Delta u(x)= \frac{1}{\mu(x)} \sum_{x \sim y} \omega_{xy}(u(y)-u(x)).
\end{equation}
Note that the kernel $p(x,y):=\frac{\omega_{xy}}{\mu(x)}$ is not a Markov kernel in general. Example \ref{ex:random} corresponds to the choice $\mu(x)=m(x)$ and $\Delta$ is  usually called the normalized Laplacian. Nevertheless, our main results are still true in this situation under suitable assumptions on $G$ that we now describe. We set for all $x \in V$, 
\begin{equation} \label{eq:dmu1}
D_\mu(x):=\frac{m(x)}{\mu(x)}
\end{equation}
and
\begin{equation} \label{eq:dmu}
D_\mu:=\sup_{x \in V} D_{\mu}(x)= \sup_{x\in V} \frac{m(x)}{\mu(x)}<\infty.
\end{equation}
Condition \eqref{eq:dmu} is equivalent to the $L^p(V,\mu)$-boundedness of the Laplace operator defined by \eqref{def:Deltabis} for all/some $p \in [1,+ \infty]$ (see \cite[Theorem 9.3]{HKLW}). \par
\noindent Assume that weights are symmetric, that is $\omega_{xy}=\omega_{yx}$ whenever $x \sim y$.  We consider two types of geometric conditions on $G$:
\begin{enumerate}
    \item Say that the graph satisfies $(H1)$ if 
\begin{enumerate}
    \item $\omega_{inf}: = \displaystyle \inf_{x \sim y} \omega_{xy}>0$,
\item $||\mu||_{\infty} := \sup_{x \in V} \mu(x)< \infty$,
  \item $D_{\mu} < \infty$.
\end{enumerate}
\item Say that the graph satisfies $(H2)$ if 
  \begin{enumerate}
      \item There exists $\alpha >0$ so that for $x$, $y \in V$, $\omega_{xy} \geq \alpha \mu(x)$
    whenever $x \sim y$ (this is condition \ref{item:A1}),
\item $D_{\mu} < \infty$.
  \end{enumerate}
\end{enumerate}
Assumption $(H1)$ is made throughout \cite{HLLY} and implies $(H2)$ with $\alpha:=\frac{\omega_{\inf}}{||\mu||_{\infty}}$. Moreover, the symmetry of $\omega_{xy}$ exactly means that the Markov chain $p$ is reversible with respect to the measure $\mu$, that is \ref{item:A2} holds. Thus,
\[
(H_1)\Rightarrow (H_2)\Rightarrow [\ref{item:A1}+\ref{item:A2}].
\]
Therefore, under (H1) or (H2), all the results given in the introduction remain true and the proofs are just adaptations of the proofs given in the rest of the paper. A straightforward argument shows that, under $(H1)$ or $(H2)$, the measure $\mu$ and the quantity $m$ are comparable. Indeed, $D_\mu<\infty$ means that there exists $C>0$ such that $m\le C\mu$. Moreover, for all $x\in V$,
\[
\mu(x)\le \frac 1{\alpha}\omega_{xy}\le \frac{m(x)}{\alpha}.
\]
Thus, the two possible 
definitions of the Laplace operator are equivalent. Assumptions \ref{item:A1} and \ref{item:A2} are strictly more general than $(H1)$ and $(H2)$, as the following example shows, and they appear in a lot of settings. 
\end{remark}
\begin{example}
Let $V:=\Z$. Define
\[
\omega_{ij}:=\left\{
\begin{array}{ll}
   \frac 1{ij}  &\mbox{ if }\left\vert j-i\right\vert=1\mbox{ and }\min\left(\left\vert i\right\vert,\left\vert j\right\vert\right)\ge 1,  \\
   1 & \mbox{ if }\left[i=0\mbox{ and }\left\vert j\right\vert=1\right]\mbox{ or }\left[j=0\mbox{ and }\left\vert i\right\vert=1\right],\\
    0 & \mbox{ otherwise }.
\end{array}
\right. 
\]
Define also
\[
\mu(i):=\left\{
\begin{array}{ll}
   \frac 1{i^4}  &\mbox{ if }i\ne 0,  \\
    1 & \mbox{if }i=0. 
\end{array}
\right.
\]
Observe first that there exists $C>0$ such that
\[
\mu\le Cm.
\]
Notice also that, for all $i\ge 2$, 
\[
\omega_{i,i+1}\le \omega_{i,i-1}\le 2\omega_{i,i+1},
\]
which entails
\[
m(i)=\omega_{i,i+1}+\omega_{i,i-1}\le 3\omega_{i,i+1}.
\]
Similarly, for all $i\le -2$,
\[
m(i)\le 3\omega_{i,i-1}.
\]
Thus, if $p(i,j):=\frac{\omega_{ij}}{\mu(i)}$ for all $i,j\in \Z$, conditions \ref{item:A1} and \ref{item:A2} hold, while 
\[
D_\mu=+\infty,
\]
showing that $(H_2)$ (and therefore $(H_1)$) are not fulfilled.
\end{example}

\end{proof}
\subsection{The Bakry-\'Emery criterion} \label{sec:BK}
Let us now turn to the Bakry-Emery condition. For all functions $f,g$ on $V$, define
\[
\Gamma_0(f,g):=fg
\]
and, for all $k\in \N$,
\[
2\Gamma_{k+1}(f,g):=\Delta\Gamma_k(f,g)-\Gamma_k(f,\Delta g)-\Gamma_k(\Delta f,g).
\]
In particular, 
\[\Gamma_1(f,g)= \frac{1}{2} \left(  \Delta (fg) - f \Delta g- g \Delta f \right),
\]
\[
\Gamma_{2}(f,g)=  \frac{1}{2} \left(  \Delta \Gamma_1(f,g) - \Gamma_1 (f , \Delta g ) - \Gamma_{1}(g, \Delta f)  \right).
\]
Denote $\Gamma_k(f):=\Gamma_k(f,f)$ and $\Gamma:=\Gamma_1$. Notice that, for all functions $f$ on $V$,
\begin{equation}\label{eq:gamma2}
2\Gamma_2f=\Delta(\Gamma f)-2\Gamma(f,\Delta f)=\Delta(\Gamma f)+f\Delta^2f+(\Delta f)^2-\Delta(f\Delta f).
\end{equation}
The operator $\Gamma$ is usually called "carr\'e du champ" operator.  \par
\noindent Let us gather here some properties of $\Gamma$ which will turn to be useful later:
\begin{lemma} \label{lem:gamma}
Let $G$ be a weighted graph with bounded geometry. For all  functions $f$, $g : V \rightarrow \R$  and all $x\in V$,
\begin{itemize}
    \item[(i)] $\Gamma(f,g)(x)= \displaystyle \frac{1}{2 } \sum_{y \sim x} p(x,y) (f(x)-f(y))(g(x)-g(y))$. In particular, 
    \begin{equation} \label{eq:exprgamma}
    \Gamma(f)(x)=\displaystyle \frac{1}{2} 
    \sum_{y \sim x} p(x,y) (f(x)-f(y))^{2}.
    \end{equation}
    \item[(ii)] Let $p\in [1,\infty]$ and $p^{\prime}$ be such that $\frac 1p+\frac 1{p^{\prime}}=1$. Then, for all $f\in L^p(V,\mu)$ and all $g\in L^{p^{\prime}}(V,\mu)$,
    \[
    \sum_{x \in V} \Gamma(f,g)(x) \mu(x)= - \sum_{x \in V} f(x) \Delta g(x) \mu(x)=- \sum_{x \in V} g(x) \Delta f(x) \mu(x).
    \]
    \item[(iii)] $\displaystyle \Delta(fg)(x) = (f \Delta g)(x) +2 \Gamma(f,g)(x)+ (g \Delta f)(x) .$
    \item[(iv)] If $f>0$ on $V$, $\displaystyle \Delta (\sqrt{f})(x)= \frac{1}{2 \sqrt{f(x)}} \Delta f(x) - \frac{1}{\sqrt{f(x)}} \Gamma(f)(x) .$
    \item[(v)] For all functions $f,g$ on $V$ and all $x\in V$,
    \begin{equation} \label{eq:gamma_f+g}
    \sqrt{\Gamma(f+g)(x)}\le \sqrt{\Gamma(f)(x)}+\sqrt{\Gamma(g)(x)}
    \end{equation}
    and
    \begin{equation} \label{eq:gamma_f-g}
    \left\vert \sqrt{\Gamma f(x)}-\sqrt{\Gamma g(x)}\right\vert\le \sqrt{\Gamma(f-g)(x)}.
    \end{equation}
    \item [(vi)] For all bounded functions $f,g$ on $V$,
    \begin{equation} \label{eq:gammaf}
    \left\Vert \Gamma f\right\Vert_{\infty}\le 2\left\Vert f\right\Vert_\infty^2
    \end{equation}
    and
    \begin{equation} \label{eq:gamma_g-f}
    \left\Vert \Gamma f-\Gamma g\right\Vert_\infty\le 2 \left(\left\Vert f\right\Vert_\infty+\left\Vert g\right\Vert_\infty\right)\left\Vert f-g\right\Vert_\infty.
    \end{equation}
\end{itemize}

\end{lemma}
\begin{remark}
On  Riemannian \Bk manifolds, the chain rule gives, for all $C^\infty$ functions $u>0$ on $M$ and all $p>0$, 
\begin{equation} \label{eq:deltaup}
\Delta u^{p} =pu^{p-1} \Delta u + \frac{p-1}{p} u^{-p} |\nabla u^{p}|^{2},
\end{equation}
and in general for the Laplace-Beltrami operator on Riemannian manifolds,  if $\phi \in C^{\infty}(\R^{+})$, $\Delta \phi(f)= \phi'(f) \Delta f + \phi''(f) \Gamma(f)$.  One cannot expect \eqref{eq:deltaup} to hold on graphs because of the lack of chain rule. However, assertion $(iv)$ in Lemma \ref{lem:gamma} shows, as noticed in \cite[Equation $(3.8)$]{BHLLMY},  that \eqref{eq:deltaup} is true for $p= 1/2$ (by replacing in a natural way $|\nabla \sqrt{f}|^{2}$ by $\Gamma(\sqrt{f})$). Fact (iv) will not be used in the sequel, but is given for possible later applications.
\end{remark}

\begin{remark}
Another way to see the lack of chain rules is to consider the  notion of diffusion operators as in \cite{BE} and in \cite{bakryledoux}. Say that $\Gamma$ is a derivation if $\Gamma$ satisfies the following chain rule:
\begin{equation} \label{eq:chainrule}
\Gamma(fg,h)= g \Gamma(f,h)+ f \Gamma(g,h).
\end{equation}
Easy computations show that $\Gamma_1$ does not satisfy \eqref{eq:chainrule} in the context of weighted graphs.
\end{remark}
\begin{proof} [Proof of Lemma \ref{lem:gamma}]
For $(i)$, it is enough to prove the second assertion, since the first one readily follows by polarization. But by definition of $\Gamma$,
\begin{eqnarray*}
2\Gamma f(x) & = & \Delta(f^2)(x)-2f(x) \Delta f(x)\\
& = & \sum_{y} p(x,y)(f^2(y)-f^2(x))-2f(x)\sum_{y} p(x,y)(f(y)-f(x))\\
& = & \sum_{y} p(x,y)(f(y)-f(x))^2\ge 0.
\end{eqnarray*}
For assertion $(ii)$, which extends formula \eqref{eq:IPP} to the case of $L^p$ and $L^{p^{\prime}}$ functions, notice first that, since $\Delta$ is $L^p$-bounded, $\Delta f\in L^p$ and
\[
\sum_{x\in V} g(x)\Delta f(x)\mu(x)  =   \sum_{x\in V}g(x)\left(\sum_{y\in V}p(x,y)(f(y)-f(x))\right) \mu(x).
\]
Observe now that, for all $x\in V$, since $p$ is a Markov kernel, 
\[
\sum_{y\in V}p(x,y)\left\vert f(y)-f(x)\right\vert \le  \left\vert f(x)\right\vert +\sum_{y\in V}p(x,y) \left\vert f(y)\right\vert,
\]
which implies
\begin{eqnarray}
\sum_{x\in V}\left\vert g(x)\right\vert \left(\sum_{y\in V}p(x,y)\left\vert f(y)-f(x)\right\vert\right) \mu(x) & \le & \sum_{x\in V}\left\vert f(x)\right\vert\left\vert g(x)\right\vert \mu(x)\nonumber \\
& +  & \sum_{x,y\in V} p(x,y) \left\vert f(y)\right\vert \left\vert g(x)\right\vert \mu(x)\nonumber\\
& = & S_1+S_2 \label{eq:doublesum}.
\end{eqnarray}
Notice now that $S_1$ and $S_2$ are finite. Indeed, on the one hand, by H\"older inequality, 
\begin{equation*}
S_1 \le  \left\Vert f\right\Vert_p \left\Vert g\right\Vert_{p^{\prime}}.
\end{equation*}
On the other hand, by H\"older inequality again, we get by \ref{item:A2}
\begin{eqnarray*}
S_2 &=  & \sum_{x,y\in V} p(x,y)^{\frac 1p} p(x,y)^{\frac 1{p^{\prime}}} 
\mu(x)^{\frac 1p} \mu(x)^{\frac 1{p^{\prime}}} 
\left\vert f(y)\right\vert \left\vert g(x)\right\vert \\
& \le & \left(\sum_{x,y\in V} \left\vert f(y)\right\vert^p p(x,y) \mu(x)\right)^{\frac 1p}\left(\sum_{x,y\in V} \left\vert g(x)\right\vert^{p^{\prime}}p(x,y) \mu(x)\right)^{\frac 1{p^{\prime}}}\\
& = & \left(\sum_{x,y\in V} \left\vert f(y)\right\vert^p p(y,x) \mu(y)\right)^{\frac 1p}\left(\sum_{x,y\in V} \left\vert g(x)\right\vert^{p^{\prime}}p(x,y) \mu(x)\right)^{\frac 1{p^{\prime}}}\\
& =& \left(\sum_{y\in V} \left\vert f(y)\right\vert^p \mu(y)\right)^{\frac 1p}\left(\sum_{x\in V} \left\vert g(x)\right\vert^{p^{\prime}}\mu(x)\right)^{\frac 1{p^{\prime}}}\\
& = & \left\Vert f\right\Vert_p\left\Vert g\right\Vert_{p^{\prime}}.
\end{eqnarray*}
Thus, the Fubini theorem applies and yields, by \ref{item:A2} again,
\begin{eqnarray*}
\sum_{x\in V} g(x)\Delta f(x)\mu(x) & =&   \sum_{x\in V}g(x)\left(\sum_{y\in V}p(x,y)(f(y)-f(x))\right) \mu(x)\\
& = & \sum_{x,y\in V}g(x)p(y,x)(f(y)-f(x))\mu(y)\\
& = & \sum_{x,y\in V}g(y)p(x,y)(f(x)-f(y)) \mu(x).
\end{eqnarray*}
As a consequence,
\begin{eqnarray*}
2\sum_{x\in V} g(x)\Delta f(x)\mu(x) & = & -\sum_{x,y\in V} p(x,y)(f(y)-f(x))(g(y)-g(x))\\
& = & -2\sum_{x\in V} \Gamma(f,g)(x)\mu(x),
\end{eqnarray*}
which is the first inequality of $(ii)$, the second one being obtained by exchanging the roles of $f$ and $g$.\par
\noindent Identity (iii) is nothing but the definition of $\Gamma$. \par
\Bk \noindent For $(iv)$, compute 
\begin{eqnarray*}
\sqrt{f(x)} \Delta (\sqrt{f})(x) &=& \sqrt{f(x)}  \sum_{y \sim x} p(x,y)
(\sqrt{f(y)} - \sqrt{f(x)})\\
& =& \sum_{y \sim x} p(x,y) (\sqrt{f(y)} \sqrt{f(x)} - f(x))\\
&=& \Delta f(x)+  \sum_{y \sim x} p(x,y) (\sqrt{f(y)}\sqrt{f(x)} -f(y)).
\end{eqnarray*}
Moreover,
\begin{eqnarray*}
\Gamma(\sqrt{f})(x) &=& \frac{1}{2} \sum_{y \sim x} p(x,y) (\sqrt{f(y)} - \sqrt{f(x)})^{2}\\
&=& \frac{1}{2} \sum_{y \sim x} p(x,y) (f(x)+f(y) - 2 \sqrt{f(y)}\sqrt{f(x)})\\
&=& - \frac{1}{2} \Delta f(x) + \sum_{y \sim x} p(x,y) (f(y)- \sqrt{f(y)}\sqrt{f(x)}). 
\end{eqnarray*}
Hence, we get 
\begin{eqnarray*}
\sqrt{f(x)} \Delta (\sqrt{f})(x)&=& \Delta f(x)+ \sum_{y \sim x} p(x,y) (\sqrt{f(y)}\sqrt{f(x)} -f(y))\\
&=& \frac{1}{2} \Delta f(x) - \Gamma(\sqrt{f})(x)) .
\end{eqnarray*} 
We now turn to the proof of (v). For all $x\in V$ and all functions $F$ defined on the set of neighbours of $x$, set
\[
\left\Vert F\right\Vert_{l^2(x)}:=\left(\sum_{y\sim x} p(x,y) \left\vert F(y)\right\vert^2\right)^{\frac 12}.
\]
For all functions $f$ on $V$ and all $x\in V$, define, for all $y\sim x$
\begin{equation} \label{eq:deftilde}
(\delta f)_x(y):=f(y)-f(x).
\end{equation}
Then, note that
\begin{eqnarray*}
\sqrt{\Gamma(f)(x)} & = & \frac 1{\sqrt{2}}\left(\sum_{y\sim x} p(x,y)\left\vert f(y)-f(x)\right\vert^2\right)^{\frac 12}\\
& = & \frac 1{\sqrt{2}} \left\Vert (\delta f)_x\right\Vert_{l^2(x)}.
\end{eqnarray*}
This entails, by the fact that $||.||_{l^{2}(x)}$ is a
norm,
\begin{eqnarray*}
\sqrt{\Gamma(f+g)(x)} & = & \frac 1{\sqrt{2}} \left\Vert (\delta (f+g))_x\right\Vert_{l^2(x)}\\
& = & \frac 1{\sqrt{2}} \left\Vert (\delta f)_x+(\delta g)_x\right\Vert_{l^2(x)}\\
& \le & \frac 1{\sqrt{2}} \left\Vert (\delta f)_x\right\Vert_{l^2(x)}+\frac 1{\sqrt{2}} \left\Vert (\delta g)_x\right\Vert_{l^2(x)}\\
& = & \sqrt{\Gamma(f)(x)}+\sqrt{\Gamma(g)(x)}.
\end{eqnarray*}
As a consequence,
\begin{eqnarray*}
\sqrt{\Gamma f(x)} & = & \sqrt{\Gamma(f-g+g)(x)}\\
& \le & \sqrt{\Gamma(f-g)(x)}+\sqrt{\Gamma g(x)},
\end{eqnarray*}
and exchanging the roles of $f$ and $g$ yields \eqref{eq:gamma_f-g}. 

\noindent Finally, let $f\in L^\infty(V,\mu)$. For all $x\in V$, since $p$ is a Markov kernel, 
\begin{eqnarray*}
\Gamma(f)(x) & = & \frac 1{2} \sum_{y\sim x} p(x,y) \left\vert f(y)-f(x)\right\vert^2\\
& \le & 2\left\Vert f\right\Vert_\infty^2 \sum_{y \sim x} p(x,y)\\
& = & 2\left\Vert f\right\Vert_\infty^2,
\end{eqnarray*}
which shows that \eqref{eq:gammaf} holds. Moreover, for all $x\in V$, by Cauchy-Schwarz inequality, 
\begin{eqnarray*}
\Gamma f(x)-\Gamma g(x) & = & \frac 1{2} \sum_{y\sim x} p(x,y) \left((f(y)-f(x))^2-((g(y)-g(x))^2\right)\\
& = &  \frac 1{2}\sum_{y\sim x} p(x,y) \left((f-g)(y)-(f-g)(x)\right)\\
& \times & \left((f+g)(y)-(f+g)(x)\right)\\
& \le &  \frac 1{2}\left(\sum_{y\sim x} p(x,y) \left((f-g)(y)-(f-g)(x)\right)^2\right)^{\frac 12}\\
& \times & \left(\sum_{y\sim x} p(x,y) \left((f+g)(y)-(f+g)(x)\right)^2\right)^{\frac 12}\\
& \le &  \sqrt{\Gamma(g-f)(x)} \times \sqrt{\Gamma(f+g)(x)} \\
& \le & 2 \left\Vert g-f\right\Vert_\infty \left(\left\Vert f\right\Vert_\infty+\left\Vert g\right\Vert_\infty\right),
\end{eqnarray*}
where the last line relies on \eqref{eq:gammaf}. Exchanging the roles of $f$ and $g$ yields \eqref{eq:gamma_g-f}.
\end{proof}
\noindent The next estimate, although elementary, will be crucial in the sequel.
\begin{proposition}\label{prop:saut}
Let $G=(V,p,\mu)$ be a weighted graph with bounded geometry. 
For any function 
$u:V \rightarrow \R$ 
so that $ ||\Gamma u ||_{\infty} \leq \alpha/2$, $|u(x)-u(y)|\leq 1$ whenever $x \sim y$.\\
\end{proposition}
\begin{proof}
If $x$, $y$ are in $V$ with $x \sim y$, we have
\begin{eqnarray*}
\alpha |u(x)-u(y)|^{2} & \leq &p(x,y) |u(x)-u(y)|^{2} \\
& \leq &  \sum_{z \sim x} p(x,z) (u(z)-u(x))^{2}\\
&\leq & 2 ||\Gamma u||_{\infty},
\end{eqnarray*}
where the first line follows from \ref{item:A1} and the last one is due to Lemma \ref{lem:gamma}. It follows that $|u(x)-u(y)| \leq 1$ if $||\Gamma u||_{\infty} \leq \alpha/2$. 
\end{proof}
\noindent Let us now introduce the curvature-dimension condition:
\begin{Def} \label{def:BE}
Let $K\in \R$ and $n\in \N^{\ast}$.
\begin{enumerate}
    \item Say that $G$ satisfies $CD(K,n)$ if and only if, for all functions $f$ on $V$,
    \begin{equation} \label{eq:CDKn} \tag{$CD(K,n)$}
\Gamma_2f\ge \frac 1n(\Delta f)^2+K\Gamma f.
\end{equation}
    \item Say that $G$ satisfies the endpoint condition $CD(K,\infty)$ if and only if, for all functions $f$ on $V$,
\begin{equation} \label{eq:CDKinfty} \tag{$CD(K,\infty)$}
\Gamma_2f\ge K\Gamma f.
\end{equation}
\end{enumerate}
\end{Def}
\begin{remark}
Note that $\mu$ plays no role in Definition \ref{def:BE}.
\end{remark}
\noindent We will consider in particular the case of graphs with nonnegative curvature, that is graphs meeting the condition:
\begin{equation} \label{eq:CD0n} \tag{$CD(0,n)$}
\Gamma_2f\ge \frac 1n(\Delta f)^2.
\end{equation}
Note that \eqref{eq:CDKn} clearly implies \eqref{eq:CDKinfty}. More generally, for all $K$, $CD(K,n_1)\Rightarrow CD(K,n_2)$ whenever $n_1\le n_2$ and for all $n$, $CD(K_1,n)\Rightarrow CD(K_2,n)$ whenever $K_1\ge K_2$.
\noindent We also mention a version  of Bonnet-Myers theorem in this setting.
\begin{theorem}
Let $G$ be weighted graph with bounded geometry satisfying 
\eqref{eq:CDKinfty} with $K>0$. Then, the diameter of $G$ is bounded by $\noindent \frac{2}{K}$.
\end{theorem}
\noindent The proof is a straightforwaed adaptation of the arguments given in \cite[Corollary 2.2]{LMP} (where the assumption is $D_{\mu} < \infty$, see remark 
\ref{rem:laplacian}) 
 and is left to the reader.  The argument also relies on some properties of the heat semigroup that are established in section \ref{sec:heat} below. It follows that under condition \eqref{eq:CDKinfty} with $K>0$,  the graph is finite. In this  case, most of our proofs could be simplified. But, we prefer to state and prove all our results in the general case $K \geq 0$.

\begin{remark}\label{rk:CDE}
In \cite{BHLLMY,HLLY} are considered some modified versions of the Bakry-\'Emery conditions, called exponential curvature-dimension inequalities $CDE(K,n)$ and $CDE'(K,n)$. Under these conditions, some of the results of the present paper are proved by considering the classical linear heat equation. If the semigroup generated by $\Delta$ is a diffusion operator (that is the operator satisfies a chain rule type formula, see \cite[Equation $(1.2)$]{bakryledoux}), $CD(K,n)$ and $CDE'(K,n)$ are equivalent (\cite[Proposition 2.1]{HLLY}) . For instance, to check that $CDE'(K,n) \Longrightarrow CD(K,n)$, one applies $CDE'(K,n)$ to $e^{f}$ to get $CD(K,n)$ for $f$. This explains why, in the Li-Yau inequality (Theorem \ref{pro:liyau}), we replace solutions of the linear heat equation by solutions of the modified one (see also Remark \ref{rk:heat-log}). In  general (that is, if the chain rule is not assumed to be true), it is not known whether conditions $CD(K,n)$ and $CDE'(K,n)$ are equivalent.
\end{remark}

\section{Bochner formulas for weighted graphs and applications to  Cayley graphs} \label{sec:abelian}

In Riemannian geometry, the Bochner formula plays an important role, related to the Ricci curvature of a manifold. For instance, a nice application 
is a proof of the Bishop-Gromov Theorem about the growth
  of volume balls on a manifold with bounds on its Ricci curvature (see \cite[Theorem 4.19]{GHL}). 
  
    \begin{theorem}[Bochner formula] \label{bochner}
   Let $(M,g)$ be a (smooth complete) Riemannian manifold. Then, if $f : M \rightarrow \R$ is smooth, one has
   \begin{equation}\label{bochnerbis}
      \frac{1}{2} \Delta |\nabla f | ^{2}= g(\nabla \Delta f, \nabla f)+ ||\mbox{Hess}(f)||_{2}^{2}+  \mbox{Ric}(\nabla f, \nabla f),
      \end{equation}
   where $\nabla f$ is the gradient of $f$ with respect to $g$, $||Hess(f)||_{2}$ is the $L^{2}$-norm of the Hessian of $f$ and $Ric$ is the Ricci curvature tensor.
 \end{theorem}
\noindent By using the terminology of Bakry-Emery, the Bochner formula can be rewritten as follows:
 \begin{equation} \label{eq:bochner}
 \Gamma_{2}(f)= Ric(\nabla f, \nabla f) + || Hess (f) ||_{2}^{2}.
 \end{equation}
\noindent Hence, the Bochner formula implies that a Riemannian manifold $M$ of dimension $n$ and nonnegative Ricci curvature satisfies $CD(0,n)$. In this section, 
  we first prove a
  Bochner formula in the general setting of weighted graphs and then in the case of Cayley graphs of finitely generated groups. 
 \begin{theorem} \label{thm:bochnergraph}
 Let $G=(V, p, \mu)$ be a weighted graph. Then, for all functions $f: V \rightarrow \R$ and all $x \in V$,
 \begin{equation}\label{bochnergraph}
     \Gamma_{2}f(x) = \frac{1}{4} |D^{2}f|^{2}(x)-\Gamma f(x)+ 
     \frac{1}{2} (\Delta f)^{2}(x)
 \end{equation}
 where
 $$ |D^{2}f|^{2}(x)= \sum_{y \sim x} p(x,y) \sum_{z\sim y} p(y,z) 
 (f(x)-2f(y)+f(z))^{2}.$$
 \end{theorem} 
 
 \begin{proof}
 For all $x\in V$,
 \begin{eqnarray*}
 \Gamma(f,\Delta f)(x) &=& \frac{1}{2} \sum_{y \sim x} p(x,y) (f(y)-f(x))(\Delta f(y)- \Delta f (x))\\
 &=& \frac{1}{2} \sum_{y \sim x} p(x,y) (f(y)-f(x)) \\
 & &\left(
 \sum_{z \sim y} p(y,z)(f(z)-f(y))- \sum_{z \sim x} p(x,z) (f(z)-f(x))
 \right)\\
 &=& \frac{1}{2} \sum_{y \sim x} p(x,y)\sum_{z \sim y} p(y,z) \left(
 (f(y)-f(x))(f(z)-f(y))
 \right) \\ 
 &-&\frac{1}{2} (\Delta f(x))^{2}.
 \end{eqnarray*}
Moreover, for all $x\in V$,
 \begin{eqnarray*}
 \Delta (\Gamma f)(x) &=& \sum_{y \sim x}p(x,y) (\Gamma f(y) - \Gamma f (x))\\
 &=& \frac{1}{2} \sum_{y \sim x} p(x,y) \sum_{z \sim y} p(y,z)(f(z)-f(y))^{2}\\
 &-& \frac{1}{2} \sum_{y \sim x} p(x,y) \sum_{z \sim x} p(x,z) (f(z)-f(x))^{2}\\
 &=& \frac{1}{2} \sum_{y \sim x} p(x,y) \sum_{z \sim y} p(y,z)(f(z)-f(y))^{2}- \Gamma f(x).
 \end{eqnarray*}
 Note that the last equality follows form the fact that $p$ is a Markov kernel.  Since $\Gamma_{2}(f)(x)= \frac{1}{2} \Delta (\Gamma f)(x)- \Gamma (f,\Delta f) (x)$, we get
 $$\Gamma_{2} f(x)= A(x) - \frac{1}{2} \Gamma f(x)+ \frac{1}{2} 
 (\Delta f (x))^{2},$$
 where, by straightforward computations,
 \begin{eqnarray*}
 A(x)&=&  \frac{1}{4} \sum_{y \sim x} p(x,y) \sum_{z \sim y} p(y,z)
 \left(  (f(z)-f(y))^{2}-2 (f(y)-f(x))(f(z)-f(y))\right)\\
 &=& \frac{1}{4}|D^{2}f|^{2}(x) - \frac{1}{2} \Gamma f(x),
 \end{eqnarray*}
from which the conclusion readily follows.
 \end{proof}
 
\noindent As an immediate consequence of Theorem \ref{thm:bochnergraph}, we get the following curvature dimension bound, valid for all weighted graphs:
\begin{corollary} \label{cor:CDgraphgen}
Any weighted graph satisfies $CD(-1,2)$.
\end{corollary}
\begin{remark}
Corollary \ref{cor:CDgraphgen} was also proved in \cite[Theorem 1.3]{LinYau}. The proof relies on the following Bochner formula (that can be obtained by easy computations): 
$$ \Delta |\nabla f|^{2}(x)= |D^{2}f|^{2}(x)- 2 |\nabla f|^{2}(x] - 2 |\Delta f|^{2}(x) +2 \langle \nabla f, \nabla \Delta f\rangle (x)   ,$$
where $|\nabla f|^{2}(x)= 2 \Gamma (f)(x)$ is the "length" of the gradient of $f$. This formula should be compared with  (\ref{bochnerbis}). 
\end{remark}
\noindent We now move to the case of Cayley graphs.
\Bk
\begin{Def} \label{def:generating}
A group $(G,.)$ is finitely generated if there exists a finite family,
called a generating family,
 $S= \{s_{1},..., s_{N}\}$
 in $G$ so that:
 \begin{enumerate}
     \item $S$ is symmetric, that is, for all $i\in \llbracket 1,N\rrbracket$, $s_i^{-1}\in S$,
     \item for any $g \in G$, there exist $M\ge 0$ and a map $\sigma:\llbracket 1,M\rrbracket\rightarrow \llbracket 1,N\rrbracket$ so that 
     \begin{equation} \label{eq:decompog}
     g= s_{\sigma(1)}...s_{\sigma(M)},
     \end{equation}
     with the convention that an empty product is equal to the identity element of $G$.
 \end{enumerate}
\end{Def}
\noindent Let $e$ be the identity element of $G$. If $g\ne e$, the infimum over $M \in \N^\ast$ such that \eqref{eq:decompog} holds  is called the length of $g$ and is denoted by $|g|$. Define also $|e|=0$. Assume that $S$ does not contain the identity element $e$ of $G$. \par

\medskip

 \noindent We use the construction explained in example \ref{ex:random}. The set of edges $\mathcal{E}$ is defined as the set of $(x,y)\in G\times G$ such that there exists 
  $s \in S$ so that $y=x.s$. Note that these edges are non oriented if we assume that $S$ is symmetric. Set $\omega_{xy}=1$ if $x \sim y$ and $\omega_{xy}=0$ otherwise. Thus, $m(x)= N$ (where $N$ is the cardinal of $S$), $p(x,y)= 1/N$ if $x \sim y$ and $p(x,y)=0$ otherwise. Define $\mu(x)=N$ for all $x \in G$. 
 Then, the weighted graph given by $(G, p, \mu)$ is called the Cayley graph of $G$ related to $S$ and has bounded geometry. Note that the set of vertices $V$ is given by all the elements of $G$. 
  In this case, the geodesic distance on the graph, also called the word metric, is given by $d(x,y)= |x^{-1}.y|$.
   In the sequel, we write $G$ instead of $(G,p,\mu)$. For more details about Cayley graphs, see for instance \cite[Section 1.5]{Meier}.\par
   \noindent Let $G$ be a finitely generated group and let $S$ be a (finite) generating family of $G$.  
 Let  $f : G \rightarrow \mathbb{R}$ and let  $x \in G$. It is convenient to use the standard Laplacian in this setting, that is  $\Delta f(x)= \sum_{y \sim x} (f(y)-f(x))$ (we drop the constant $1/N$).\par
 \noindent Let us compute
 $\Gamma_2(f) = \frac{1}{2}(\Delta\Gamma(f) -2 \Gamma(f,\Delta f))$. We first express the Laplacian of  $f$, and then $\Gamma(f,\Delta f)$. We will 
 use the notation 
 \[
 \partial_{i}f(x) =f(x.s_{i})-f(x)
 \]
for all $i\in \llbracket 1,N\rrbracket$ and all $x \in G$.

    \begin{proposition} \label{bochner1}
      For all $x \in G$, all $f$, $g :G  \rightarrow \R$ and all $i\in \llbracket 1,N\rrbracket$, one has
      \begin{enumerate}
      \item \label{Deltaf} $\displaystyle \Delta f(x)= \sum_{j=1}^{N} \partial_{j} f(x)$.
      \item \label{difg} $\partial_{i} (fg)(x)= \partial_{i}f(x)g(xs_{i}) + f(x) \partial_{i} g(x)$.
        \item \label{Gammafg} $\displaystyle \Gamma(f,g)(x)= \frac{1}{2} \sum_{j=1}^{N} (\partial_{j} f(x) \partial_{j} g(x))$. In particular,
        $\displaystyle \Gamma(f)= \frac{1}{2} \sum_{j=1}^{N} (\partial_{j} f (x))^{2}$.
        \end{enumerate}
      \end{proposition}

\noindent Note that \eqref{difg} shows that $\partial_{i}$ is not a derivation because it does not satisfy the Leibniz rule. 

      \begin{proof}
        All the proofs are straightforward. First, 
        $$\Delta f(x)= \sum_{y \sim x} (f(y)-f(x))= \sum_{j=1}^{N} (f(xs_{j})-f(x))= \sum_{j=1}^{N} \partial_{j}f(x).$$
        To prove  \eqref{difg}, note that
        \begin{eqnarray*}
          \partial_{i}(fg)(x) &= &(fg)(xs_{i})-(fg)(x)\\
                            &=& (f(xs_{i}) -f(x))g(xs_{i}) + f(x)(g(xs_{i})-g(x))\\
                            &=& \partial_{i}f(x)g(xs_{i}) + f(x) \partial_{i} g(x).
        \end{eqnarray*}
        From  \eqref{Deltaf} and \eqref{difg}, we deduce
        \begin{eqnarray*}
          \Gamma(f,g)(x) &=&  \frac{1}{2} (\Delta(fg)(x)-f(x) \Delta g(x)-g(x) \Delta f(x))\\
                      & =& \frac{1}{2} \sum_{j=1}^{N} \left(\partial_{j}f(x)g(xs_{j})+ f(x) \partial_{j}g(x)-f(x) \partial_{j}g(x)-g(x)
                           \partial_{j}f(x)      \right)\\
                      &=& \frac{1}{2} \sum_{j=1}^{N} \partial_{j}f(x)(g(xs_{j})-g(x))\\
          &=& \frac{1}{2} \sum_{j=1}^{N} \partial_{j}f(x) \partial_{j}g(x)
        \end{eqnarray*}
        and the proof of \eqref{Gammafg} is complete.
      \end{proof}
\noindent The point is now to prove an analog of \eqref{eq:bochner} in the case of Cayley graphs.
\begin{theorem} \label{bochner2}
With the same notations as above, for every  $x \in G$ and every  $f:G \rightarrow \R$,
        \begin{eqnarray*}
          \Gamma_{2}(f)(x) &=& \frac{1}{4} \sum_{i=1}^{N} \sum_{j=1}^{N} (\partial_{i}\partial_{j} f(x))^{2}\\
          &+& \frac{1}{2}
      \sum_{i=1}^{N} \sum_{j=1}^{N}
      \partial_{j}f(x)( \partial_{i}\partial_{j}f(x)- \partial_{j}\partial_{i}f(x))
      \end{eqnarray*}
      or by symmetrization
      \begin{equation} \label{eq:gamma2cayley}
      \begin{aligned}
      \Gamma_{2}(f)(x)&= \frac{1}{4} \sum_{i=1}^{N} \sum_{j=1}^{N} (\partial_{i}\partial_{j} f(x))^{2}\\
        &- \frac{1}{4}
      \sum_{i=1}^{N} \sum_{j=1}^{N}
            (\partial_{i}f(x)- \partial_{j}f(x))( \partial_{i}\partial_{j}f(x)- \partial_{j}\partial_{i}f(x)).
      \end{aligned}
      \end{equation}
          \end{theorem}
          By analogy with \eqref{eq:bochner}), the first term in the right-hand side of \eqref{eq:gamma2cayley} should be considered as the $L^{2}$-norm of the Hessian of $f$ and the second one as the Ricci curvature related to  $\nabla f$. 
          \begin{proof}
 From Proposition \ref{bochner1}, we first obtain
            $$\Gamma(f \Delta f)(x)= \frac{1}{2} \sum_{i=1}^{N} \partial_{i}f(x) \partial_{i}(\Delta f)(x)=
            \frac{1}{2} \sum_{i=1}^{N} \sum_{j=1}^{N} \partial_{i}f(x) \partial_{i} \partial_{j}f(x),$$
            then
            
            \begin{eqnarray*}
              \Delta \Gamma(f)(x) &=& \frac{1}{2} \sum_{i=1}^{N} \sum_{j=1}^{N} \partial_{i} (\partial_{j}f)^{2}(x)\\
                                  &=& \frac{1}{2} \left( \sum_{i=1}^{N} \sum_{j=1}^{N}  (\partial_{i}\partial_{j}f(x))( \partial_{j}f(x)) +
                                      \sum_{i=1}^{N} \sum_{j=1}^{N} (\partial_{i} \partial_{j} f)(x) (\partial_{j}f)(xs_{i}) \right).
            \end{eqnarray*}
            Hence, we conclude
            \begin{eqnarray*}
\Gamma_{2}(f)(x) &=& \frac{1}{2} (\Delta\Gamma(f)(x)-2 \Gamma(f,\Delta f)(x))\\
             &=&\frac{1}{4} \left(\sum_{i=1}^{N} \sum_{j=1}^{N}  (\partial_{i}\partial_{j}f(x))( \partial_{j}f(x)) +
                 \sum_{i=1}^{N} \sum_{j=1}^{N} (\partial_{i} \partial_{j} f)(x) (\partial_{j}f)(xs_{i}) \right)\\
              &-& \frac 12 \sum_{i=1}^{N} \sum_{j=1}^{N} \partial_{i}f(x) \partial_{i} \partial_{j} f(x)\Bk\\
              &=& \frac{1}{4} \left( \sum_{i=1}^{N} \sum_{j=1}^{N}  (\partial_{i}\partial_{j}f(x))( \partial_{j}f(x)) +
                                      \sum_{i=1}^{N} \sum_{j=1}^{N} (\partial_{i} \partial_{j} f)(x) (\partial_{j}f)(xs_{i}) \right)\\
                 &-& \frac{1}{2} \sum_{i=1}^{N} \sum_{j=1}^{N} \partial_{j}f(x) \partial_{i} \partial_{j}f(x) + \frac{1}{2} \sum_{i=1}^{N}
                     \sum_{j=1}^{N} \partial_{j}f(x) (\partial_{i} \partial_{j} f(x)- \partial_{j} \partial_{i} f(x))\\
                 &=& \frac{1}{4} \left( \sum_{i=1}^{N} \sum_{j=1}^{N} \partial_{i} \partial_{j}f(x) (\partial_{j}f(xs_{i}) -
                     \partial_{j}f(x))   \right)\\
              &+& \frac{1}{2}
      \sum_{i=1}^{N} \sum_{j=1}^{N}
                  \partial_{j}f(x)( \partial_{i}\partial_{j}f(x)- \partial_{j}\partial_{i}f(x))\\
              &=& \frac{1}{4} \sum_{i=1}^{N} \sum_{j=1}^{N} (\partial_{i}\partial_{j} f(x))^{2}\\
          &+& \frac{1}{2}
      \sum_{i=1}^{N} \sum_{j=1}^{N}
      \partial_{j}f(x)( \partial_{i}\partial_{j}f(x)- \partial_{j}\partial_{i}f(x)).
            \end{eqnarray*}
            
          \end{proof}
          
\noindent Note that we did not use the fact that $S$ is symmetric in the previous proof. The symmetry of $S$, on the other hand, will play an important role in the proof of  the next result, which is a consequence of the Bochner formula.
         
           \begin{theorem}\label{bochner3}
           With the same notations as above, for every  $x \in G$ and every  $f:G \rightarrow \R$, the following lower bound holds:
        \begin{eqnarray*}
          \Gamma_{2}(f)(x) &\geq &  \frac{1}{2N} \Delta f(x)^{2}\\
          &+& \frac{1}{2}
      \sum_{i=1}^{N} \sum_{j=1}^{N}
      \partial_{j}f(x)( \partial_{i}\partial_{j}f(x)- \partial_{j}\partial_{i}f(x))
      \end{eqnarray*} 
           \end{theorem}
           
          \begin{proof}
            To simplify the notations, write 
            $S= \{e_{1}^{+},...,e_{N}^{+}, e_{1}^{-}, ...,e_{N}^{-} \}$ with $e_{i}^{-}=(e_{i}^{+})^{-1}$. For any $i\in \llbracket 1,N\rrbracket$,
            we set $\partial_{i}^{+}f(x)= f(xe_{i}^{+})-f(x)$ and $\partial_{i}^{-}f(x)=f(xe_{i}^{-})-f(x)$ whenever $x \in G$ and
            $f: G \rightarrow \R$.  Note that
            $$\partial_{i}^{+} \partial_{i}^{-}f(x)= f(x e_{i}^{+} e_{i}^{-})-f(xe_{i}^{+}) -f(xe_{i}^{-})+f(x)=
            - (\partial_{i}^{+}f(x) + \partial_{i}^{-}f(x)).$$
            For the same reasons, $\partial_{i}^{-} \partial_{i}^{+}f(x)= - (\partial_{i}^{+}f(x) + \partial_i^-f(x)\Bk)
            = \partial_{i}^{+} \partial_{i}^{-} f(x)$. Hence, the Cauchy-Schwarz inequality entails
            \begin{eqnarray*}
              \frac{1}{4} \sum_{i=1}^{2N} \sum_{j=1}^{2N} (\partial_{i}\partial_{j} f(x))^{2} & \geq & \frac{1}{4} \sum_{i=1}^{N} \left(
                 \partial_{i}^{+} \partial_{i}^{-}f(x)^{2}+ \partial_{i}^{-} \partial_{i}^{+} f(x)^{2}
                                     \right)\\
              & = & \frac{1}{2} \sum_{i=1}^{N} (\partial_{i}^{+}f(x)+ \partial_{i}^{-}f(x))^{2}\\
                            & \geq& \frac{1}{2 N } \left( \sum_{i=1}^{N} (\partial_{i}^{+}f(x)+ \partial_{i}^{-}f(x))
                                    \right)^{2}\\
              & =& \frac{1}{2N} \Delta f(x)^{2},
            \end{eqnarray*}
           and the conclusion follows from Theorem \ref{bochner2}.          \end{proof}
          
          \noindent In the sequel, we will use the following observation:  $\partial_{i}\partial_{j}f(x) - \partial_{j}\partial_{i}f(x)= f(xe_{i}e_{j}) - f(xe_{j}e_{j})$. The next result is a straightforward application of Theorem \ref{bochner3}.
          \begin{corollary}
          Let $G$ be an Abelian group equipped with a symmetric generating family $(e_{1}{+}$, ...,$e_{N}^{+}$, $e_{1}^{-1}$,..., $e_{N}^{-})$. Then, $G$ satisfies $CD(0,2N)$.
          \end{corollary}
           \noindent This result improves  \cite[Theorem 2.3]{KKRT}  where it is proved that the Cayley graph of an Abelian group satisfies $CD(0,\infty)$. Let us now consider a nonabelian situation.
           
          \begin{corollary}
          Let $S_{n}$ be the symmetric group of all bijections of $\llbracket 1,n\rrbracket$ onto itself. Assume that the generating family is given by the set $\mathcal{T}_{n}$ of all transpositions. Then, $S_{n}$ satisfies $CD(0,n(n-1))$. 
\end{corollary}
\begin{proof}
Fix $x \in S_{n}$, $(ij) \in \mathcal{T}_{n}$ and define
$S= \sum_{\tau \in \mathcal{T}_{n}} f(x\tau (ij)- f(x (ij) \tau) $. By using the change of variables $\tau \rightarrow (ij)\tau(ij)$, we get that $S= \sum_{\tau \in \mathcal{T}_{n}}
f(x (ij)\tau)-f(x \tau (ij))=-S$. Thus, $S=0$ and the conclusion follows from Theorem \ref{bochner3}.
\end{proof}

\noindent In \cite{KKRT}, the Ricci curvature of a graph is defined as the maximum value of the $K$ so that $CD(K, \infty)$ holds. Then, it is proved that the Ricci curvature of the symmetric group $S_{n}$ ($n \geq 3$) is $2$. It seems quite difficult to compare their result with ours.

  \begin{remark}
  
  Let $G=(V,E)$ be a discrete graph. If $v \in V$, denote by $N(v)$ the number of neighbours of $v$. Let $d\ge 1$ be an integer. Say that $G$ is $d$- regular if and only if $N(v)=d$ for all $v \in V$.\par
  \noindent Assume that $G$ is $d$-regular. Let $v\in V$ and ${\mathcal N}(v)$ be the set of neighbours of $v$. Say that $G$ has a local $d$-frame at $v$ if and only if there are mappings $\eta_{1}$, ..., $\eta_{d} : {\mathcal N}(v) \rightarrow V$ so that:
  \begin{enumerate}
  \item for all $i\in \llbracket 1,d\rrbracket$, $v\sim \eta_{i} (v)$,
  \item for all $i \neq j$, $\eta_{i}(v) \neq \eta_{j}(v)$,
  \item for all $i\in \llbracket 1,d\rrbracket$,
 $$ \bigcup_{j=1}^{d} \eta_{i}\eta_{j}(v)= \bigcup_{k=1}^{d} \eta_{k} \eta_{i}(v).$$
 \end{enumerate}
 Say that $G$ is Ricci flat of dimension $d$ (see \cite[Section 4]{chungyau}) if
  \begin{enumerate}
  \item $G$ is connected;
  \item $G$ is $d$ regular,
  \item for all $v\in V$, $G$ has a local $d$-frame at $v$.
  \end{enumerate}
 The Cayley graph of an Abelian group is Ricci flat. A Ricci flat graph of dimension $d$  satisfies $CD(0,d)$. The proof is a  very easy modification of the proof for Cayley graphs of Abelian groups, using the maps $\eta_i$.
  \end{remark}        
       
\section{The heat semigroup} \label{sec:heat}
Throughout all this section, we consider a weighted graph $G=(V, p, \mu)$ with bounded geometry.
Let $p\in [1,+\infty]$. By Lemma \ref{lem:propdelta}, $\Delta$ is bounded in $L^p(V,\mu)$, Thus, $\Delta$ generates a uniformly continuous semigroup on $L^p(V,\mu)$ (\cite[Chapter 1, Theorem 1.2]{P}), denoted by $(P_t)_{t\ge 0}$. Recall that this means that $P_{0}=I$ (where $I$ denotes the identity operator), $P_{t+s}= P_{t} \circ P_{s}$ for every $t, s \geq 0$,  and $\displaystyle \lim_{t \rightarrow 0} || P_{t}-I||_{p}=0$. Moreover, $t \rightarrow P_{t}$ is diffentiable and $\displaystyle \frac{d}{dt} P_{t} = \Delta P_{t}= P_{t} \Delta$. Note also that if $f \geq 0$, $P_{t}f \geq 0$. See \cite[Chapter 1, section 1.1]{P}.\par

\medskip

\noindent Notice that, since, for all $f\in L^\infty(V,\mu)$,
\[
P_tf=\sum_{n=0}^\infty \frac{(t\Delta)^nf}{n!},
\]
one has
\begin{equation} \label{eq:completestoch},
P_t{\bf 1}_V={\bf 1}_V,
\end{equation}
where ${\bf 1}_V$ is the constant function on $V$ equal to $1$ everywhere. Identity \eqref{eq:completestoch} means that the weighted graph $G=(V,p,\mu)$ is stochastically complete. This property has many alternative characterizations, see \cite[Theorem 3.3]{Wo}).\par

\medskip

\noindent In the sequel, we provide  \Bk estimates for $\Gamma P_t$ and $\sqrt{\Gamma P_{t}}$ under the Bakry-Emery  curvature/dimension  conditions. These estimates will turn to be crucial to solve \eqref{eq:heatmodif1} by means of semigroup methods.
\begin{proposition} \label{pro:estimgrad}
Assume that $G$ satisfies \eqref{eq:CDKn} with $K\ge 0$ and $n \in \N$. Then, for all functions $f\in L^\infty(V,\mu)$ and all $t>0$,
\begin{equation}\label{eq:gammapt1}
\Gamma(P_{t}f)\leq e^{-2Kt} P_{t} (\Gamma f)- \frac{2}{n} \int_{0}^{t} e^{-2Ks} P_{s}(P_{t-s} \Delta f)^{2}ds.
\end{equation}
Moreover, 
\begin{equation}\label{eq:gammapt2}
    \Gamma(P_{t}f) \leq e^{-2Kt} P_{t}(\Gamma f)+ \frac{e^{-2Kt}-1}{Kn} (P_{t} \Delta f)^{2},  
\end{equation}
if $K>0$ and
\begin{equation}\label{eq:gammapt2bis}
    \Gamma(P_{t}f) \leq P_{t}(\Gamma f)- \frac{2t}{n} (P_{t} \Delta f)^{2}, 
\end{equation}
if $K=0$.
\end{proposition}
\begin{proof}
The argument is reminiscent of the proof of \cite[Lemma 4.1]{KKRT}. Let $f\in L^\infty(V,\mu)$. 
Note that $\Gamma f\in L^\infty(V)$. Let $t>0$ and define 
\[
g_s(x):=e^{-2Ks}P_s(\Gamma(P_{t-s}f))(x)
\]
for all $s\in [0,t]$ and all $x\in V$. Observe first that
\[
2\Gamma (P_{t-s}f)=\Delta((P_{t-s}f)^2)-2P_{t-s}f\Delta (P_{t-s}f),
\]
which entails
\[
\frac{\partial}{\partial s}\Gamma (P_{t-s}f)=-\Delta(P_{t-s}f\Delta P_{t-s}f)+(\Delta P_{t-s}f)^2+(P_{t-s}f)\Delta(\Delta P_{t-s}f).
\]
As a consequence, for all $s\in [0,t]$,
\begin{eqnarray*}
\frac{\partial}{\partial s}P_s(\Gamma(P_{t-s}f)) & = & P_s\Delta(\Gamma(P_{t-s}f))+  P_s\frac{\partial}{\partial s}(\Gamma(P_{t-s}f))\\
& = & P_s\Big(\Delta \Gamma(P_{t-s}f)-\Delta(P_{t-s}f\Delta P_{t-s}f)\\
& & +(\Delta P_{t-s}f)^2+(P_{t-s}f)\Delta(\Delta P_{t-s}f)\Big)\\
& = & P_s(2\Gamma_2(P_{t-s}f)),
\end{eqnarray*}
where the last line follows from \eqref{eq:gamma2}. As a consequence, by using \eqref{eq:CDKn}, we get
\begin{eqnarray*}
\frac{\partial g}{\partial s} & = & e^{-2Ks}(-2KP_s(\Gamma(P_{t-s}f))+P_s(2\Gamma_2(P_{t-s}f)))\\
& = & 2e^{-2Ks}P_s(\Gamma_2(P_{t-s}f)-K\Gamma(P_{t-s}f))\\
& \geq & 2e^{-2Ks} P_{s} \left( \frac{1}{n} (\Delta P_{t-s}f)^{2}\right).
\end{eqnarray*}
By integrating between $0$ and $t$, we get 
\begin{eqnarray*}
e^{-2Kt} P_{t}(\Gamma f) - \Gamma (P_{t}f) \geq \frac{2}{n} \int_{0}^{t} e^{-2Ks} P_{s}((P_{t-s} \Delta f)^{2}) ds
\end{eqnarray*}
which is exactly the inequality \eqref{eq:gammapt1}. In particular, \eqref{eq:gammapt1} yields 
\begin{equation} \label{eq:comparptgamma}
e^{-2Kt} P_{t}(\Gamma f)\ge \Gamma (P_{t}f).
\end{equation}
\begin{remark}\label{modif1}
Note that, if we assume  that $G$ satifies \eqref{eq:CDKinfty}, a slight modification of the end of the previous proof gives \eqref{eq:comparptgamma}.
\end{remark}
\noindent Let us now turn to the proof of \eqref{eq:gammapt2}. Observe first that, for all bounded functions $f$ on $V$ and all $t>0$,
\begin{equation} \label{eq:ptf2}
P_t(f^2)-(P_t f)^2\ge 2t\Gamma(P_tf)\ge 0.
\end{equation}
Indeed, \Bk fix $t>0$ and define, for all $s\in [0,t]$,
\[
g_s:=P_s((P_{t-s}f)^2).
\]
An easy computation shows that
\begin{eqnarray*}
\frac{\partial g}{\partial s} & = & P_s\Delta((P_{t-s}f)^2)-2P_s(P_{t-s}f\Delta P_{t-s}f)\\
& = & P_s\Big(\Delta((P_{t-s}f)^2)-2P_{t-s}f\Delta P_{t-s}f\Big)\\
& = & 2P_s(\Gamma P_{t-s}f),
\end{eqnarray*}
and \eqref{eq:comparptgamma} yields
\[
\frac{\partial g}{\partial s}\ge 2e^{2Ks}\Gamma(P_sP_{t-s}f)=2e^{2Ks}\Gamma(P_tf),
\]
and, since $K\ge 0$, integrating over $[0,t]$ therefore induces \eqref{eq:ptf2}. \par
\noindent By the properties of the heat semigroup $P_{s}$, we therefore have 
$$ P_{s}(P_{t-s} \Delta f)^{2} \geq (P_{s}(P_{t-s}\Delta f ))^{2}= (P_{t} \Delta f)^{2}.$$
As a consequence, if $K>0$,
\begin{eqnarray*}
   \Gamma (P_{t} f) &\leq & e^{-2Kt} P_{t}(\Gamma f)- \frac{2}{n} (P_{t} \Delta f)^{2} \int_{0}^{t} e^{-2 Ks} ds \\
   &=& e^{-2Kt} P_{t}(\Gamma f) + \frac{e^{-2Kt} -1}{Kn} (P_{t} \Delta f)^{2}
    \end{eqnarray*}
    while, if $K=0$, one obtains
    \begin{eqnarray*}
   \Gamma (P_{t} f) &\leq &  P_{t}(\Gamma f)- \frac{2}{n} (P_{t} \Delta f)^{2} \int_{0}^{t}  ds \\
   &=& P_{t}(\Gamma f) - \frac{2t}{n} (P_{t} \Delta f)^{2}.
    \end{eqnarray*}
    The proofs of \eqref{eq:gammapt2} and \eqref{eq:gammapt2bis} are thus complete.
\medskip
\end{proof}



\begin{corollary}
Assume that the weighted graph $G$ satisfies \eqref{eq:CDKinfty} with $K\geq 0$. Then, we have  for all functions $f\in L^\infty(V,\mu)$ and all $t>0$,
\begin{equation} \label{eq:gammapt}
\Gamma(P_tf)\le e^{-2Kt}P_t(\Gamma f)
\end{equation}
\begin{equation} \label{eq:gradpt0}
\left\Vert \Gamma P_tf\right\Vert_\infty\le \left\Vert \Gamma f\right\Vert_\infty
\end{equation}
and
\begin{equation} \label{eq:gradpt}
\left\Vert \sqrt{\Gamma(P_tf)}\right\Vert_\infty\le \frac 1{\sqrt{t}} \left\Vert f\right\Vert_\infty. 
\end{equation}
\end{corollary}

\begin{proof}

\noindent As already seen in Remark \ref{modif1}, \eqref{eq:gammapt} holds under assumption \eqref{eq:CDKinfty}. As a consequence of \eqref{eq:gammapt}, we get 
\begin{eqnarray*}
\left\Vert \Gamma P_tf\right\Vert_\infty & \le & \left\Vert P_t\Gamma f\right\Vert_\infty\\
& \le & \left\Vert \Gamma f\right\Vert_\infty,
\end{eqnarray*}
which is \eqref{eq:gradpt0}.\par

\noindent Finally, \eqref{eq:ptf2}, the proof of which remains valid under \eqref{eq:CDKinfty}, implies
\[
\Gamma(P_tf)\le \frac 1t P_t(f^2)\le \frac 1t \left\Vert f\right\Vert_\infty^2,
\]
which in turn shows that \eqref{eq:gradpt} holds. 
\end{proof}
\begin{remark}
Using \eqref{eq:ptf2}, one can even prove in the same way that
\[
\left\Vert \sqrt{\Gamma(P_tf)}\right\Vert_\infty\le \frac 1{\sqrt{2t}}\left\Vert f\right\Vert_\infty
\]
\end{remark}
\section{The modified heat equation} \label{sec:modifiedheat}
\noindent This section is devoted to the Cauchy problem for the modified heat equation on weighted graphs with bounded geometry. We first focus on the linear heat equation.
\subsection{The Cauchy problem for the linear heat equation}
\noindent If $T>0$ and $u:[0,T)\rightarrow L^p(V,\mu)$ is a function, say that $u$ solves the (linear) heat equation if and only if $u$ is continuous on $[0,T)$, $C^1$ on $(0,T)$ and, for all $t\in (0,T)$,
\[
\frac{du}{d t}=\Delta u(t).
\]
We can use the semigroup $(P_t)_{t\ge 0}$ to solve the inhomogeneous initial value problem for the heat equation. More precisely, let $p\in [1,\infty]$. Given $T>0$, $u_0\in L^p(V,\mu)$ and $f:[0,T)\rightarrow L^p(V,\mu)$, we look for $u:[0,T) \rightarrow L^p(V,\mu)$ solving
\begin{equation} \label{eq:Cauchy}
\left\{
\begin{array}{ll}
\frac{du}{dt}=\Delta u(t)+f(t) & \mbox{ for all }t\in (0,T),\\
u(0)=u_0.
\end{array}
\right.
\end{equation}
Our claim is as follows:
\begin{proposition} \label{pro:inhomogcauchy}
Let $p\in [1,\infty]$ and $T>0$. Let $f\in L^1((0,T):L^p(V,\mu))$ be a continuous function on $(0,T]$. Then, for all $u_0\in L^p(V,\mu)$, the Cauchy problem \eqref{eq:Cauchy} has a unique solution $u$, continuous on $[0,T)$, $C^1$ on $(0,T)$ and given by
\[
u(t)=P_tu_0+\int_0^t P_{t-s}f(s)ds.
\]
\end{proposition}
\begin{proof}
Since $\Delta$ is bounded on $L^p(V,\mu)$, it follows that 
\[
\int_0^T \left\Vert \Delta f(s)\right\Vert_p ds\le C\int_0^T \left\Vert f(s)\right\Vert_p ds<\infty,
\]
and the conclusion follows \cite[Chapter 4, Corollary 2.2 and Corollary 2.6]{P}.
\end{proof}

\subsection{The Cauchy problem for the modified heat equation}
Let $T>0$, $p\in [1,\infty]$. A function $u:[0,T)\rightarrow L^p(V,\mu)$ is called a solution of the modified heat equation if $u$ is continuous on $[0,T)$, $C^1$ in $(0,T)$ and is a solution of
\begin{equation}\label{eq:heatmodif}
\frac{du(t)}{dt}=\Delta u(t)+\Gamma u(t)
\end{equation}
for all $t\in (0,T)$. We now prove the following theorem.
\begin{theorem} \label{pro:solheatmodif}
Assume that $G$ is a weighted graph with bounded geometry satisfying \eqref{eq:CDKinfty} for some $K \geq 0$.
Let $u_0\in L^\infty(V,\mu)$ (non constant) and set
\[
T= \frac{1}{256} \left\Vert \Gamma u_0\right\Vert_\infty^{-1}.
\]
Then, there exists a unique solution $u:[0,T)\rightarrow L^\infty(V,\mu)$ of \eqref{eq:heatmodif} such that $u(0)=u_0$ and, for all $t\in [0,T)$,
\begin{equation} \label{eq:conduniq}
\left\Vert \sqrt{\Gamma u(t)}\right\Vert_\infty\le 2\left\Vert \sqrt{\Gamma u_0}\right\Vert_\infty.
\end{equation}
\end{theorem}
\begin{remark}
\begin{enumerate}
\item The inequality \eqref{eq:conduniq} is required to ensure uniqueness of the solution. It is unclear whether there exists a unique solution $u$ of \eqref{eq:heatmodif} such that $u(0)=u_0$ if \eqref{eq:conduniq} is not required.
\item Since condition \eqref{eq:CDKn} implies \eqref{eq:CDKinfty}, the theorem remains true for a weighted graph $G$ satisfying \eqref{eq:CDKn} for some $n \in \N^{*}$.
\end{enumerate}
\end{remark}
\begin{proof}
The argument is inspired by the proof of \cite[Theorem 1.1]{AB}. We  construct a sequence of functions $(u^k)_{k\ge -1}:[0,T)\rightarrow L^\infty(V,\mu)$, continuous on $[0,T)$ and $C^1$ on $(0,T)$, such that $u^{-1}=0$, solving, for all $k\in\N$, 
\begin{equation} \label{eq:cauchyuk}
\left\{
\begin{array}{ll}
\frac{du^k}{dt}=\Delta u^k+\Gamma u^{k-1} & \mbox{ for all }t\in (0,T),\\
u^k(0)=u_0,
\end{array}
\right.
\end{equation}
and satisfying
\begin{equation} \label{eq:estimk}
M_{k}:=\sup_{0\le t<T} \left\Vert \sqrt{\Gamma u^{k}(t)}\right\Vert_\infty\le 2\left\Vert \sqrt{\Gamma u_0}\right\Vert_\infty
\end{equation}
for all $k\ge -1$. Assume that, for some $k\in \N$, $u^{k-1}:[0,T)\rightarrow L^\infty$ is continuous on $[0,T)$, $C^1$ on $(0,T)$ and satisfies \eqref{eq:estimk} with $k-1$ instead of $k$. Then, $t\mapsto \Gamma u^{k-1}(t)$ is bounded on $[0,T)$, therefore integrable on $(0,T)$. Moreover, since 
\[
2\Gamma u^{k-1}=\Delta ((u^{k-1})^2)-2u^{k-1}\Delta u^{k-1},
\]
assertion $(iii)$ in Lemma \ref{lem:propdelta} ensures that $t\mapsto \Gamma u^{k-1}(t)$ is continuous on $[0,T)$. Therefore, Proposition \ref{pro:inhomogcauchy} shows that \eqref{eq:cauchyuk} has a unique solution $u^k:[0,T)\rightarrow L^\infty$, continuous on $[0,T)$ and $C^1$ on $(0,T)$, given by
\begin{equation} \label{eq:uk}
u^k(t)=P_tu_0+\int_0^t P_{t-s}(\Gamma u^{k-1}(s))ds.
\end{equation}
Let us now check that \eqref{eq:estimk} holds for $k$. Set $F(t) : = P_{t} u_{0}$ and 
$G_{k-1} (t)= \displaystyle \int_{0}^{t} P_{t-s}(\Gamma u^{k-1}(s)) dt $ so that  $u^{k}(t) = F(t)+G_{k-1}(t)$. By \eqref{eq:gamma_f+g}, we have 
\begin{equation} \label{eq:split}
\sqrt{\Gamma (u^{k}(t))} \leq \sqrt{\Gamma F(t)}+ \sqrt{\Gamma G_{k-1}(t)},
\end{equation}
where the inequality holds in the pointwise sense on $V$. 
By \eqref{eq:gradpt0}, 
\begin{equation} \label{eq:gammaF}
||\Gamma F(t)||_{\infty} \leq ||\Gamma u_{0}||_{\infty}.
\end{equation}
To estimate $\sqrt{\Gamma G_{k-1}(t)}$, we first claim that, for all continuous and integrable functions $\theta:[0,T)\rightarrow L^\infty$, if
\[
g(t):=\int_0^t \theta(s)ds,
\]
then
\begin{equation} \label{eq:sqrtgamma}
\sqrt{\Gamma g(t)} \leq  \int_{0}^{t} \sqrt{\Gamma \theta(s)} ds.
\end{equation}
Indeed, let $x\in V$. Using the notation introduced in \eqref{eq:deftilde}, for all $y\in V$ with $x\sim y$,
\[
(\delta g(t))_x(y)=\int_{0}^{t} (\delta\theta(s))_x(y)ds.\footnote{Here, we use the fact that, if $s\mapsto F(s)$ is a continuous function from $[0,t]$ to $L^\infty(V)$ and if $G:=\int_0^t F(s)ds$, then, for all $x\in V$, $G(x)=\int_0^t F(s)(x)ds$. Indeed, it is plain to see that this identity holds true for step functions. The general case follows with uniform approximation by step functions.}
\]
Thus, arguing as in the proof of Lemma \ref{lem:gamma} (v),  
\begin{eqnarray*}
\sqrt{\Gamma (g(t))(x)} & = & \frac{1}{\sqrt{2}} ||(\delta g(t))_{x} ||_{l^2(x)}\\
  & = & \frac{1}{\sqrt{2} } \left | \left| \int_{0}^{t}  (\delta \theta(s))_{x} ds \right|\right|_{l^{2}(x)}\\
  & \le & \frac{1}{\sqrt{2 }} \int_{0}^{t} \left | \left|   (\delta\theta(s))_{x} \right|\right|_{l^{2}(x)}ds \\
  & = & \int_{0}^{t} \sqrt{\Gamma \theta(s)} ds.
\end{eqnarray*}
Hence, \eqref{eq:sqrtgamma} is proved. As a consequence, one obtains
\begin{eqnarray}
\sqrt{\Gamma G_{k-1}(t)} & \leq & \int_{0}^{t} \sqrt{\Gamma P_{t-s}(\Gamma u^{k-1}(s))} ds \nonumber \\
            & \leq & \int_{0}^{t} \frac{1}{\sqrt{t-s}} || \Gamma u^{k-1}(s)||_{\infty} ds \nonumber \\
               & \leq& 2 \sqrt{t} \sup_{0\le s<T}|| \Gamma u^{k-1}(s)||_{\infty} \nonumber\\
               & \leq& 8\sqrt{T} || \Gamma u_{0}||_{\infty} \le ||\Gamma u_{0}||_{\infty}^{1/2}, \label{eq:gammaG} 
\end{eqnarray}
where the second line relies on \eqref{eq:gradpt} in Proposition \ref{pro:estimgrad} and the fourth one is due to \eqref{eq:estimk} applied with $k-1$. Gathering \eqref{eq:split}, \eqref{eq:gammaF} and \eqref{eq:gammaG} yields \eqref{eq:estimk} for $k$. \par
\noindent Let us now establish that the sequence $(\Gamma u^k(t))_{k\ge -1}$ converges in $L^\infty(V,\mu)$ uniformly in $t\in [0,T)$. For all $k \geq -1$, define
\[
N_{k}(T):=\sup_{0\le s<T}|| \Gamma u^{k+1}(s)- \Gamma u^{k}(s)||_{\infty}.
\]
and notice that $N_k(T)<\infty$ for all $k\ge -1$ by \eqref{eq:estimk}. Then, for any $s \in [0,T)$, we have 
\begin{eqnarray*}
|| \Gamma u^{k+1}(s) - \Gamma u^{k}(s)||_{\infty} & = & \left\Vert \left(\sqrt{\Gamma u^{k+1}(s)}-\sqrt{\Gamma u^k(s)}\right) \left(\sqrt{\Gamma u^{k+1}(s)}+\sqrt{\Gamma u^k(s)}\right)\right\Vert_\infty\\
& \leq  & 4 
||\sqrt{\Gamma u_{0}}||_{\infty} \left\Vert \sqrt{\Gamma u^{k+1}(s)} - \sqrt{\Gamma u^{k}(s)}\right\Vert_{\infty},
\end{eqnarray*}
where the second line is due to \eqref{eq:estimk}. Moreover,
\begin{eqnarray*}
\left\vert \sqrt{\Gamma u^{k+1}(s)} -\sqrt{\Gamma u^{k}(s)}\right\vert &\le & \sqrt{\Gamma (u^{k+1}(s)-u^{k}(s))}\\
 &=& \sqrt{\Gamma \left( \int_{0}^{t} P_{t-s}(\Gamma u^{k}(s)- \Gamma u^{k-1}(s)) ds \right)}\\
 & \leq & \int_{0}^{t} \sqrt{\Gamma P_{t-s}(\Gamma u^{k}(s)- \Gamma u^{k-1}(s))} ds\\
 & \leq & 2\sqrt{t} N_{k-1}(T),
\end{eqnarray*}
where the first line follows from \eqref{eq:gamma_f-g}, the third one relies on \eqref{eq:sqrtgamma} and the last one on \eqref{eq:gradpt}. Finally, we obtain 
\[
N_{k}(T)\leq 8 || \sqrt{\Gamma u_{0}}||_{\infty} \sqrt{T} N_{k-1}(T)= \frac 12 N_{k-1}(T).
\]
This shows the existence of $v\in L^\infty(0,T)$ such that 
\[
\lim_{k\rightarrow+\infty} \left\Vert \Gamma u^k(t)-v(t)\right\Vert_\infty=0
\]
and the convergence is uniform on $[0,T)$. Therefore
$$
\lim_{k\rightarrow+\infty} \int_0^t P_{t-s}\Gamma u^k(s)ds=\int_0^t P_{t-s}v(s)ds
$$
uniformly on $[0,s]$ for all $s\in [0,T)$. Thus, \eqref{eq:uk} entails that the sequence $(u^k)_{k\ge -1}$ converges uniformly on $[0,s]$ for all $s\in [0,T)$ to a function $u$ satisfying
\[
u(t)=P_tu_0+\int_0^t P_{t-s}v(s)ds.
\]
Since $u^k(t)\rightarrow u(t)$ pointwise on $V$, the explicit expression of $\Gamma$ given by \eqref{eq:exprgamma} shows that $\Gamma u^k(t)\rightarrow \Gamma u(t)$ pointwise on $V$, and since $\Gamma u^k(t)\rightarrow v$ pointwise on $V$, $v=\Gamma u$. Thus
\begin{equation} \label{eq:formule_u}
u(t)=P_tu_0+\int_0^t P_{t-s}\Gamma u(s)ds.
\end{equation}
Since $s\mapsto u(s)$ is continuous in $[0,T)$ (as a uniform limit on $[0,t]$ for all $t\in [0,T)$ of a sequence of continuous functions), so are $t\mapsto \int_0^t P_{t-s}\Gamma u(s)$ (by \eqref{eq:formule_u}) and $t\mapsto \Delta \int_0^t P_{t-s}\Gamma u(s)$ (by continuity of $\Delta$ on $L^\infty(V)$). As a consequence, \cite[Chapter 4, Theorem 2.4]{P} shows that there exists a unique continuous function $w:[0,T)\rightarrow L^\infty$, $C^1$ on $(0,T)$, solution of 
\[
\frac{dw}{dt}=\Delta w(t)+\Gamma u(t)
\]
such that $w(0)=u_0$, and that $w$ is given by
\[
w(t)=P_tu_0+\int_0^t P_{t-s}\Gamma u(s)ds.
\]
As a consequence \eqref{eq:formule_u} shows that $w=u$. This ends the proof of the existence part in Theorem \ref{pro:solheatmodif}.\par

\medskip

\noindent As far as uniqueness is concerned, let $u:[0,T)\rightarrow L^\infty(V,\mu)$ and $v:[0,T)\rightarrow L^\infty(V,\mu)$ be continuous in $[0,T)$, $C^1$ in $(0,T)$, solving
\[
\frac{du}{dt}=\Delta u+\Gamma u,\ \frac{dv}{dt}=\Delta v+\Gamma v
\]
in $[0,T)$, with $u(0)=v(0)=u_0$ and
\begin{equation} \label{eq:estimgradu}
\left\Vert \sqrt{\Gamma u(t)}\right\Vert_\infty\le 2\left\Vert \sqrt{\Gamma u_0}\right\Vert_\infty
\end{equation}
and
\begin{equation} \label{eq:estimgradv}
\left\Vert \sqrt{\Gamma v(t)}\right\Vert_\infty\le 2\left\Vert \sqrt{\Gamma u_0}\right\Vert_\infty.
\end{equation}

\noindent Since $s\mapsto \Gamma u(s)$ and $s\mapsto\Gamma v(s)$ are integrable on $(0,t)$ for all $t\in [0,T)$, \cite[Chapter 4, Corollary 2.2]{P} entails
\begin{equation} \label{eq:formulau}
u(t)=P_tu_0+\int_0^t P_{t-s}\Gamma u(s)ds \mbox{ for all }t\in[0,T),
\end{equation}
and
\begin{equation} \label{eq:formulav}
v(t)=P_tu_0+\int_0^t P_{t-s}\Gamma v(s)ds \mbox{ for all }t\in[0,T).
\end{equation}

\noindent Define now $w:=u-v$. Identities \eqref{eq:formulau} and \eqref{eq:formulav} imply
\begin{equation} \label{eq:exprw}
w(t)=\int_0^t P_{t-s}(\Gamma u(s)-\Gamma v(s))ds.
\end{equation}
Identity \eqref{eq:exprw} yields as before
\begin{eqnarray}
\sqrt{\Gamma w(t)} & \le & \int_0^t \sqrt{\Gamma P_{t-s}(\Gamma u(s)-\Gamma v(s))}ds \nonumber\\
& \le & 2\sqrt{t}\sup_{0\le s\le t} \left\Vert \Gamma u(s)-\Gamma v(s)\right\Vert_\infty. \label{eq:estimgammaw}
\end{eqnarray}
For all $s\in [0,t]$,
\begin{eqnarray}
\left\Vert \Gamma u(s)-\Gamma v(s)\right\Vert_\infty & \le & 4\left\Vert \sqrt{\Gamma u_0}\right\Vert_\infty \left\Vert \sqrt{\Gamma u(s)}-\sqrt{\Gamma v(s)}\right\Vert_\infty \nonumber\\
& \le & 4\left\Vert \sqrt{\Gamma u_0}\right\Vert_\infty \left\Vert \sqrt{\Gamma (u-v)(s)} \right\Vert_\infty\nonumber\\
& = & 4\left\Vert \sqrt{\Gamma u_0}\right\Vert_\infty \left\Vert \sqrt{\Gamma w(s)} \right\Vert_\infty \label{eq:gammau-gammav},
\end{eqnarray}
where the second line uses \eqref{eq:gamma_f-g}. Going back to \eqref{eq:estimgammaw}, we conclude that
\[
\sqrt{\Gamma w(t)} \le 8\sqrt{t}\left\Vert \sqrt{\Gamma u_0}\right\Vert_\infty \sup_{0\le s<T}\left\Vert \sqrt{\Gamma w(s)} \right\Vert_\infty,
\]
therefore, if $A:=\sup_{0\le s<T} \left\Vert \sqrt{\Gamma w(s)}\right\Vert_\infty$ (note that $A<\infty$ by \eqref{eq:estimgradu} and \eqref{eq:estimgradv}), we deduce
\[
A\le 8A\sqrt{T}\left\Vert \sqrt{\Gamma u_0}\right\Vert_\infty \le \frac A2,
\]
which entails that $A=0$, therefore proving that $w(t)$ is constant for all $t\in [0,T)$. Then \eqref{eq:gammau-gammav} implies that $\Gamma u=\Gamma v$, and \eqref{eq:exprw} shows that $w=0$, so that $u=v$. 
\end{proof}
\begin{remark} \label{rem:uniqueness}
The previous argument also shows that, if $\beta>0$ and $u$ and $v$ are solutions of \eqref{eq:heatmodif} on $[0,\beta)$ such that $u(0)=v(0)=u_0$ and
\[
\left\Vert \sqrt{\Gamma u(t)}\right\Vert_\infty\le 2\left\Vert \sqrt{\Gamma u_0}\right\Vert_\infty\mbox{ for all }t\in [0,\beta)
\]
and
\[
\left\Vert \sqrt{\Gamma v(t)}\right\Vert_\infty\le 2\left\Vert \sqrt{\Gamma u_0}\right\Vert_\infty \mbox{ for all }t\in [0,\beta),
\]
then $u(t)=v(t)$ for all $t\in [0,\beta)$.
\end{remark}
\subsection{Global solutions}
\noindent Let us now prove Theorem \ref{thm:solmodif}, which states that,  provided that $\left\Vert \Gamma u_0\right\Vert_\infty$ is small enough, there exists a unique solution of \eqref{eq:heatmodif} such that $u(0)=u_0$ and that this solution is defined on $[0,\infty)$ (note that \eqref{eq:conduniq} is not required in the statement of Theorem \ref{thm:solmodif}). Since the condition \eqref{eq:CDKn} implies \eqref{eq:CDKinfty}, the theorem remains true for a weighted graph $G$ satisfying \eqref{eq:CDKn} for some $n \in \N^{*}$.
\begin{proof}
Let $T=\frac 1{256}\left\Vert \Gamma u_0\right\Vert_\infty^{-1}$ and $u:[0,T)\rightarrow L^\infty$ be the solution of \eqref{eq:heatmodif} such that $u(0)=u_0$ and
\[
\left\Vert \sqrt{\Gamma u(t)}\right\Vert_\infty\le 2\left\Vert \sqrt{\Gamma u_0}\right\Vert_\infty.
\]
We claim that, for all $t\in [0,T)$, \eqref{eq:estimgammaut} holds. The arguments are inspired by the proof of \cite[Theorem 2.1]{M}, with several modifications due to the fact that $V$ may be infinite. Let $\varepsilon \in (0,1)$ such that
\begin{equation} \label{eq:conditioneps}
\varepsilon (T+1)<\alpha/2 -\left\Vert \Gamma u_0\right\Vert_{\infty}.
\end{equation}
Fix $x_0\in V$ and define $B_j:=B(x_0,j)$ for all integer $j\ge 1$. For every $j\ge 1$, consider the set
$$
A_j:=\left\{t\in [0,T);\ \left\Vert \Gamma u(t)\right\Vert_{L^\infty(B_j)}\ge e^{-2Kt}\left\Vert \Gamma u_0\right\Vert_{L^\infty(V)}+\varepsilon t+\varepsilon\right\}.
$$
Note that $A_j\subset A_{j+1}$ for any $j\ge 1$. Assume that $A_{j_0}\ne \emptyset$ for some $j_0\ge 1$, so that $A_j\ne \emptyset$ for all $j\ge j_0$. Since $s\mapsto \left\Vert \Gamma u(s)\right\Vert_{L^\infty(B_j)}$ is continuous, we can consider $t_j:=\inf A_j>0$ for all $j\ge j_0$. The sequence $(t_j)_{j\ge j_0}$ is nonnegative and nonincreasing, and therefore converges to some $t\ge 0$. For all $j\ge j_0$,
\begin{equation} \label{eq:tj}
\left\Vert \Gamma u(t_j)\right\Vert_{L^\infty(B_j)}= e^{-2Kt_j}\left\Vert \Gamma u_0\right\Vert_{L^\infty(V)}+\varepsilon t_j+\varepsilon.
\end{equation}
Moreover, 
\begin{eqnarray*}
\left\vert \left\Vert \Gamma u(t)\right\Vert_{L^\infty(V)}-\left\Vert \Gamma u(t_j)\right\Vert_{L^\infty(B_j)}\right\vert & \le &  \left\Vert \Gamma u(t)\right\Vert_{L^\infty(V)}-\left\Vert \Gamma u(t)\right\Vert_{L^\infty(B_j)}\\
& + & \left\vert \left\Vert \Gamma u(t)\right\Vert_{L^\infty(B_j)}-\left\Vert \Gamma u(t_j)\right\Vert_{L^\infty(B_j)}\right\vert\\
& \le &  \left\Vert \Gamma u(t)\right\Vert_{L^\infty(V)}-\left\Vert \Gamma u(t)\right\Vert_{L^\infty(B_j)}+\left\Vert \Gamma u(t)-\Gamma u(t_j)\right\Vert_{L^\infty(B_j)}\\
& \le & \left\Vert \Gamma u(t)\right\Vert_{L^\infty(V)}-\left\Vert \Gamma u(t)\right\Vert_{L^\infty(B_j)}+\left\Vert \Gamma u(t)-\Gamma u(t_j)\right\Vert_{L^\infty(V)},
\end{eqnarray*}
and since $\Gamma u(t_j)\rightarrow \Gamma u(t)$ uniformly in $V$, one deduces
\begin{equation} \label{eq:convutj}
\lim_{j\rightarrow +\infty} \left\Vert \Gamma u(t_j)\right\Vert_{L^\infty(B_j)}= \left\Vert \Gamma u(t)\right\Vert_{L^\infty(V)}.
\end{equation}
As a consequence of \eqref{eq:tj} and \eqref{eq:convutj},
\[
\left\Vert \Gamma u(t)\right\Vert_{L^\infty(V)}= e^{-2Kt}\left\Vert \Gamma u_0\right\Vert_{L^\infty(V)}+\varepsilon t+\varepsilon,
\]
which shows that $t>0$. \par
\noindent There exists $j_1\ge j_0$ such that, for all $j\ge j_1$, 
\begin{equation} \label{eq:maxgammaut}
\left\Vert \Gamma u(t)\right\Vert_{L^\infty(V)}-\left\Vert \Gamma u(t)\right\Vert_{L^\infty(B_j)}\le \frac{\varepsilon}{12}
\end{equation}
and
\begin{equation} \label{eq:ututj}
\left\vert \Gamma u(t)(x)-\Gamma u(t_j)(x)\right\vert\le \frac{\varepsilon}{12}\mbox{ for all }x\in V.
\end{equation}
Combining \eqref{eq:maxgammaut} and \eqref{eq:ututj}, one obtains
\begin{eqnarray}
\left\Vert \Gamma u(t_j)\right\Vert_{L^\infty(V)}-\left\Vert \Gamma u(t_j)\right\Vert_{L^\infty(B_j)} & = & \left\Vert \Gamma u(t_j)\right\Vert_{L^\infty(V)}-\left\Vert \Gamma u(t)\right\Vert_{L^\infty(V)} \nonumber\\
& + & \left\Vert \Gamma u(t)\right\Vert_{L^\infty(V)}-\left\Vert \Gamma u(t)\right\Vert_{L^\infty(B_j)} \nonumber\\
& + & \left\Vert \Gamma u(t)\right\Vert_{L^\infty(B_j)}-\left\Vert \Gamma u(t_j)\right\Vert_{L^\infty(B_j)}\nonumber\\
& \le & \frac \varepsilon{4}. \label{eq:gammautj}
\end{eqnarray}
For all $j\ge j_1$, pick up $x_j\in B_j$ such that 
\[
\Gamma u(t_j)(x_j)=\left\Vert \Gamma u(t_j)\right\Vert_{L^\infty(B_j)}.
\]
For all $j\ge j_1$,
\begin{eqnarray*}
\frac{\partial \Gamma u(t)}{\partial t}\vert_{t=t_j}(x_j) & = & 2\Gamma\left(u(t_j),\frac{\partial u(t)}{\partial t}\vert_{t=t_j}\right)(x_j)\\
& = & 2\Gamma\left(u(t_j),\Delta u(t_j)+\Gamma u(t_j)\right)(x_j)\\
& = & -2\Gamma_2u(t_j)(x_j)+\Delta \Gamma u(t_j)(x_j)+2\Gamma (u(t_j),\Gamma u(t_j))(x_j)\\
& \le & -2K \Gamma u(t_j)(x_j)+\Delta \Gamma u(t_j)(x_j)+2\Gamma(u(t_j),\Gamma u(t_j))(x_j),
\end{eqnarray*}
where the last inequality follows from condition \eqref{eq:CDKinfty}.\par
\noindent An easy computation yields
\begin{eqnarray} 
\Delta \Gamma u(t_j)(x_j)+2\Gamma(u(t_j),\Gamma u(t_j))(x_j) & = &  \sum_{y\in V} p(x_j,y) \left(\Gamma u(t_j)(y)-\Gamma u(t_j)(x_j)\right) \nonumber\\
& \times & \left(1+u(t_j)(y)-u(t_j)(x_j)\right) \label{eq:deltagamma}.
\end{eqnarray}
Observe that $\Gamma u(t_j)(x_j)=\left\Vert \Gamma u(t_j)\right\Vert_{L^\infty(B_j)}<\alpha/2$, since, by \eqref{eq:tj} and condition \eqref{eq:conditioneps},
\begin{eqnarray*}
\left\Vert \Gamma u(t_j)\right\Vert_{L^\infty(B_j)} & = & e^{-2Kt_j}\left\Vert \Gamma u_0\right\Vert_{\infty}+\varepsilon t_j+\varepsilon \\
& \le & \left\Vert \Gamma u_0\right\Vert_{\infty}+\varepsilon(T+1)<\alpha/2.
\end{eqnarray*}
Proposition \ref{prop:saut} entails that $\left\vert u(t_j)(y)-u(t_j)(x_j)\right\vert\le 1$ whenever $y\sim x_j$.\par
\noindent Let $y\sim x_j$. If $\Gamma u(t_j)(y)\le \Gamma u(t_j)(x_j)$, since $1+u(t_j)(y)-u(t_j)(x_j)\ge 0$, 
\begin{equation} \label{eq:maj1}
(\Gamma u(t_j)(y)-\Gamma u(t_j)(x_j))(1+u(t_j)(y)-u(t_j)(x_j))\le 0.
\end{equation}
If $\Gamma u(t_j)(y)>\Gamma u(t_j)(x_j)$, since $1+u(t_j)(y)-u(t_j)(x_j)\le 2$, \eqref{eq:gammautj} yields
\begin{equation} \label{eq:maj2}
(\Gamma u(t_j)(y)-\Gamma u(t_j)(x_j))(1+u(t_j)(y)-u(t_j)(x_j))\le \frac{\varepsilon}2.
\end{equation}
One therefore deduces from \eqref{eq:maj1}, \eqref{eq:maj2} and \eqref{eq:deltagamma} that
$$
\Delta \Gamma u(t_j)(x_j)+2\Gamma(u(t_j),\Gamma u(t_j))(x_j)\le \frac{\varepsilon}{2},
$$
therefore
\begin{equation} \label{eq:majorderiv}
\frac{\partial \Gamma u(t_j)}{\partial t}\vert_{t=t_j}(x_j)\le -2K \Gamma u(t_j)(x_j)+\frac{\varepsilon}{2}.
\end{equation}
In another connection, for all $s\in [0,t_j)$,
\begin{eqnarray*}
\frac{\Gamma u(s)(x_j)-\Gamma u(t_j)(x_j)}{s-t_j} & \ge & \frac{\left\Vert \Gamma u(s)\right\Vert_{L^\infty(B_j)}-\left\Vert \Gamma u(t_j)\right\Vert_{L^\infty(B_j)}}{s-t_j} \\
& \ge & \frac{e^{-2Ks}\left\Vert \Gamma u_0\right\Vert_{\infty}+\varepsilon s+\varepsilon-e^{-2Kt_j}\left\Vert \Gamma u_0\right\Vert_{\infty}-\varepsilon t_j-\varepsilon}{s-t_j}\\
& = & \left\Vert \Gamma u_0\right\Vert_{\infty} \frac{e^{-2Ks}-e^{-2Kt_j}}{s-t_j}+\varepsilon,
\end{eqnarray*}
and letting $s$ go to $t_j$ yields 
\begin{eqnarray*}
\frac{\partial \Gamma u(t)}{\partial t}\vert_{t=t_j}(x_j)  & \ge &  -2Ke^{-2Kt_j}\left\Vert \Gamma u_0\right\Vert_\infty+\varepsilon\\
& \ge & -2K\left\Vert \Gamma u(t_j)\right\Vert_{L^\infty(B_j)}+\varepsilon.
\end{eqnarray*}
Comparing with \eqref{eq:majorderiv}, we get
$$-2K\left\Vert \Gamma u(t_j)\right\Vert_{L^\infty(B_j)}+\varepsilon \leq -2K\left\Vert \Gamma u(t_j)\right\Vert_{L^\infty(B_j)}+\varepsilon/2,$$
which is impossible, therefore showing that $A_j=\emptyset$ for all $j\ge 1$, which means that, for all $t\in [0,T)$ and all $j\ge 1$,
$$
\left\Vert \Gamma u(t)\right\Vert_{L^\infty(B_j)}\le e^{-2Kt}\left\Vert \Gamma u_0\right\Vert_{\infty}+\varepsilon t+\varepsilon.
$$
Letting $j\rightarrow +\infty$ and then $\varepsilon$ go to $0$ provides the inequality \eqref{eq:estimgammaut}. \par
\noindent Since $\left\Vert \Gamma u\left(\frac T2\right)\right\Vert_\infty\le \left\Vert \Gamma u_0\right\Vert_\infty$, there exists a unique solution $v:\left[\frac T2,T\right]$ of
\[
\frac{dv(t)}{dt}=\Delta v(t)+\Gamma v(t)
\]
such that $v\left(\frac T2\right)=u\left(\frac T2\right)$. By iteration, the solution $u$ can therefore be extended to $[0,\infty)$ (see the argument after \cite[(2.14)]{AB}). \par
\noindent Let $u_0\in L^\infty(V)$ such that $\left\Vert \Gamma u_0\right\Vert_\infty<\frac{\alpha}2$. Consider $u,v:[0,\infty)\rightarrow L^\infty(V,\mu)$, solutions of \eqref{eq:heatmodif} such that $u(0)=v(0)=u_0$. The previous proof ensures that, for all $T>0$ and all $t\in [0,T)$, 
\[
\left\Vert \sqrt{\Gamma u(t)}\right\Vert\le \left\Vert \sqrt{\Gamma u_0}\right\Vert_\infty,\ \left\Vert \sqrt{\Gamma v(t)}\right\Vert\le \left\Vert \sqrt{\Gamma u_0}\right\Vert_\infty,
\]
and uniqueness in Theorem \ref{pro:solheatmodif} shows that $u(t)=v(t)$ for all $t\in [0,T)$. Since this holds for all $T>0$, $u(t)=v(t)$ for all $t\ge 0$.
\end{proof}
\section{Semigroup comparisons} \label{sec:compar}
We start with an easy result. 
\begin{lemma} \label{lem:oscillut}
  Let $G$ be a weighted graph with bounded geometry satisfying \eqref{eq:CDKinfty} and let $u_0\in L^\infty(V,\mu)$ such that
  \begin{equation} \label{eq:sizegammau_0}
  ||\Gamma u_{0} ||_{\infty} <  \displaystyle \frac{\alpha}{2}.
  \end{equation}
  Let $u:[0,\infty)\rightarrow L^\infty(V,\mu)$  be the solution of \eqref{eq:heatmodif} such that $u(0)=u_0$ given by Theorem \ref{thm:solmodif}. Then, for all $t \geq 0$, and
  $x \in V$, if
  $y \sim x$,
  $|u(t)(y)-u(t)(x)| \leq 1$.
\end{lemma}
\begin{proof}
  Theorem \ref{thm:solmodif} ensures that $\left\Vert \Gamma u(t)\right\Vert_\infty<\frac{\alpha}2$ for all $t\ge 0$ and Proposition \ref{prop:saut} yields the conclusion.
  \end{proof}
\noindent Let us derive from Lemma \ref{lem:oscillut} the following pointwise comparison involving the heat semigroup and solutions of the modified heat equation:
\begin{proposition} \label{pro:semigroupcompar}
There exist $\gamma_0,\gamma_1>0$ with $0<\gamma_0\le\gamma_1$ satisfying the following properties: for all weighted graphs $G$ with bounded geometry satisfying \eqref{eq:CDKinfty}, all $u_0\in L^\infty(V,\mu)$ satisfying \eqref{eq:sizegammau_0}:
\begin{enumerate}
\item for all $\gamma \ge \gamma_1$ and all $t >0$,  
\[
P_te^{\gamma u_0}\ge e^{\gamma u(t)},
\]
\item for all $\gamma\le \gamma_0$, for all $t>0$,
\[
P_te^{\gamma u_0}\le e^{\gamma u(t)}.
\]
\end{enumerate}
In these statements, $u:[0,\infty)\rightarrow L^\infty(V,\mu)$ is the solution of \eqref{eq:heatmodif} such that $u(0)=u_0$. 
\end{proposition}
\begin{proof}
We follow \cite[Theorem 2.3]{M}. Let $t >0$ and define
\[
G(s):=P_{t-s}e^{\gamma u(s)}
\]
for $s\in (0,t)$. A simple calculation gives
\begin{eqnarray*}
\frac{\partial G}{\partial s} & = & -\Delta P_{t-s}e^{\gamma u(s)}+\gamma P_{t-s}e^{\gamma u(s)}\frac{\partial u(s)}{\partial s}\\
& = & P_{t-s}\left(\gamma e^{\gamma u(s)}\frac{\partial u(s)}{\partial s}-\Delta e^{\gamma u(s)}\right)\\
& = & P_{t-s}\left(\gamma e^{\gamma u(s)}\left(\Delta u(s)+\Gamma u(s)\right)-\Delta e^{\gamma u(s)}\right)\\
& =: & P_{t-s}v(s).
\end{eqnarray*}
Let $x\in V$. Set $w(y):=u(s)(y)-u(s)(x)$ for all $y\sim x$, then
\begin{eqnarray*}
v(s)(x) & = & e^{\gamma u(s)(x)} \left(\gamma \sum_{y\sim x} p(x,y)(w(y)+\frac 12 w^2(y))\right)\\
& - & e^{\gamma u(s)(x)} \sum_{y\sim x} p(x,y)  \left(e^{\gamma w(y)}-1\right)\\
& = & e^{\gamma u(s)(x)} \sum_{y\sim x} p(x,y) \left(\gamma\left(w(y)+\frac 12 w^2(y)\right)-\left(e^{\gamma w(y)}-1\right)\right).
\end{eqnarray*}
Define 
\[
g_\gamma(z):=\gamma z (1+z/2)-(e^{\gamma z} -1)
\]
for all $z\in \R$. Since, by Lemma \ref{lem:oscillut}, $\left\vert w(y)\right\vert\le 1$ for all $y\sim x$, it is enough to prove that there exist  $0<\gamma_0<\gamma_1$ such that, for all $\gamma\in [0,\gamma_0]$ (resp. all $\gamma\ge \gamma_1$) and all $z\in [-1,1]$, $g_\gamma(z)\ge 0$ (resp. $g_\gamma(z)\le 0$). \par
\noindent Let $z\in [-1,1]\setminus \left\{0\right\}$. The Taylor Lagrange formula yields the existence of $c\in ]0,\gamma z[$ if $z>0$ (resp. $c\in ]\gamma z,0[$ if $z<0$) such that
\[
e^{\gamma z}=1+\gamma z+\frac{\gamma^2z^2}2e^{c},
\]
which shows that
\[
\frac{\gamma^2z^2}2\le e^{\gamma z}-1-\gamma z\le \frac{\gamma^2z^2}2e^{\gamma}\mbox{ if }0<z\le 1
\]
and
\[
\frac{\gamma^2z^2}2e^{\gamma z}\le e^{\gamma z}-1-\gamma z\le \frac{\gamma^2z^2}2\mbox{ if }-1\le z<0.
\]
Thus, for all $z\in (0,1]$,
\begin{eqnarray*}
g_\gamma(z) & \ge & \gamma z (1+z/2)-\left(\gamma z+\frac{\gamma^2z^2}2e^{\gamma}\right)\\
& = & \gamma\frac{z^2}2(1-\gamma e^\gamma),
\end{eqnarray*}
while for all $z\in [-1,0)$,
\begin{eqnarray*}
g_\gamma(z) & \ge & \gamma z (1+z/2)-\left(\gamma z+\frac{\gamma^2z^2}2\right)\\
& = & \gamma\frac{z^2}2(1-\gamma).
\end{eqnarray*}
Therefore, $g_\gamma(z)\ge 0$ for all $z\in [-1,1]$ whenever $\gamma e^{\gamma}\le 1$, that is when $0\le \gamma\le \gamma_0:=\Omega$, where $\Omega$ is the ``$\Omega$ constant'', which is the unique solution of $\Omega e^{\Omega}=1$ ($\Omega=0.567...$).\par
\noindent On the other hand, for all $z\in (0,1]$,
\begin{eqnarray*}
g_\gamma(z) & \le & \gamma z (1+z/2)-\left(\gamma z+\frac{\gamma^2z^2}2\right)\\
& = & \gamma\frac{z^2}2(1-\gamma)\le 0
\end{eqnarray*}
when $\gamma\ge 1$. Assume now that $\gamma\ge 2$. For all $z\in [-1,-\frac{2}{\gamma}]$,
\begin{eqnarray*}
g_\gamma(z) & \le & \gamma z (1+z/2)+1\\
& \le & \frac{\gamma}2 z+1\le 0.
\end{eqnarray*}
If $-\frac{2}{\gamma}<z<0$,
\[
e^{\gamma z}-1\ge \gamma z+\frac{\gamma^2z^2}2e^{-2},
\]
so that
\[
g_\gamma(z)\le \gamma\frac{z^2}2\left(1-\gamma e^{-2}\right)\le 0
\]
whenever $\gamma \ge \gamma_1:=e^2$. 
\end{proof}

\section{Li-Yau and Harnack type inequalities} \label{sec:liyau}

The main goal of this section is to prove variants of the Li-Yau and Harnack inequalities for solutions of the modified heat equation (under the Bakry-Emery condition)  which are classical in geometric analysis.
\subsection{A Li-Yau inequality}
Recall that on a $n$-dimensional manifold $M$ with non-negative Ricci curvature, the Li-Yau inequality could be written as
$$ - \Delta_{M} \log P^{M}_{t}f \leq \frac{n}{2t}$$
 for all positive functions  $f$ on $M$, where $P^{M}_{t}=e^{t\Delta_{M}}$ is the classical heat semigroup and $\Delta_{M}$ is the Laplace-Beltrami operator on $M$ (see \cite{LY}). In the sequel, we prove an analogous inequality (Theorem \ref{pro:liyau}) for graphs under $CD(0,n)$, replacing in a natural way $\log P^{M}_{t}f$ by a solution $u_{t}$ of the modified heat equation (see Remark \ref{rk:heat-log}).\par


\begin{proof}
We follow arguments of the proof of \cite[Theorem 3.1]{M} with modifications due to the fact that $V$ may be infinite. Set $F(t)= t \Delta u(t)$ for $t>0$. Fix a basepoint $v\in V$ and, for all integer $j\ge 1$, let $B_j:=B(v,j)$. Let $\varepsilon,T>0$. For all $j\ge 1$, the map $t\mapsto \min_{B(v,j)} F(t)$ is well-defined (since $B(v,j)$ is a finite set) and continuous on $[0,T]$ and therefore reaches its minimum at some $t_j\in [0,T]$. Pick up $x_j\in B_j$ such that
\begin{equation} \label{eq:minbj}
m_j:=\Delta u(t_j)(x_j)=\min_{B_j} \Delta u(t_j).
\end{equation}
There exists $t\in [0,T]$ such that, up to a subsequence, $\lim_{j\rightarrow+\infty} t_j=t$. Since $\Delta u(t)$ is bounded on $V$, there exists $j_0\ge 1$ such that, for all $j\ge j_0$,
\begin{equation} \label{eq:infdeltau}
\inf_{B_j} \Delta u(t)-\inf_V \Delta u(t)<\frac{\varepsilon}8.
\end{equation}
As a consequence, there exists $j_1\ge 1$ such that, for all $j\ge j_1$,
\begin{equation} \label{eq:infdeltauj}
\inf_{B_j} \Delta u(t_j)-\inf_V \Delta u(t_j)<\frac{\varepsilon}2.
\end{equation}
Indeed, there exists $j_2\ge 1$ such that, for all $j\ge j_2$ and all $x\in V$,
\begin{equation} \label{eq:deltau-uj}
\left\vert \Delta u(t)(x)-\Delta u(t_j)(x)\right\vert\le \frac{\varepsilon}8.
\end{equation}
Let $j_1:=\max(j_0,j_2)$ and $j\ge j_1$. 
One has
\begin{eqnarray}
\inf_{B_j} \Delta u(t_j) & \le & \frac{\varepsilon}8+\inf_{B_j} \Delta u(t)\nonumber\\
& \le & \frac{\varepsilon}4+\inf_V \Delta u(t)\nonumber\\
& = & \frac{\varepsilon}2+\inf_{V} \Delta u(t_j)\label{eq:infbj}.
\end{eqnarray}

\medskip

\noindent Let $j\ge j_1$. Assume that $m_{j}<0$, so that $t_j>0$. As a consequence,
  \begin{eqnarray*}
    0 \geq \partial_{t} F(t)\vert_{t=t_j}(x_j) & =& t_j  \partial_{t} \Delta  u(t)\vert_{t=t_j}(x_j)+ \Delta u(t_j)(x_j)\\
                                  & =& t_j \Delta \partial_{t} u(t)\vert_{t=t_j}(x_j)+ \Delta u(t_j)(x_j) \\
                                   &=& t_j(\Delta \Delta u(t_j)(x_j)+ \Delta \Gamma u(t_j)(x_j)) +\Delta u(t_j)(x_j) \\
    & \geq & t_j \left(\Delta \Delta u(t_j)(x_j)+2 \Gamma(u(t_j), \Delta u(t_j))(x)+ \frac{2}{n} (\Delta u(t_j)(x_j))^{2}\right) + \Delta u(t_j)(x_j),
  \end{eqnarray*}
  where the second line holds since $\partial_{t} \Delta = \Delta \partial_{t}$, the third one follows from the modified heat equation satisfied by $u$ and the last inequality is due to $CD(0,n)$ since
  $$ \Gamma_{2}(u(t_j))(x_j) = \frac{1}{2} \left(   \Delta \Gamma(u(t_j))(x_j) -2 \Gamma(u(t_j), \Delta u(t_j))(x_j)\right) \geq \frac{1}{n} (\Delta u(t_j)(x_j))^{2}.$$
  On the other hand, we have
  \begin{eqnarray*}
    \Delta \Delta u(t_j)(x_j)+2 \Gamma(u(t_j), \Delta u(t_j))(x_j) & =&  \sum_{y \sim x_j} p(x_j,y)  (\Delta u(t_j)(y)- \Delta u(t_j)(x))\\
                                                            &+&  \sum_{y \sim x_j} p(x_j,y) (\Delta u(t_j)(y)- \Delta u(t_j)(x_j))\\
                                                            & \times & (u(t_j)(y)-u(t_j)(x_j))\\
    &=&  \sum_{y \sim x} p(x_j,y) (\Delta u(t_j)(y)- \Delta u(t_j)(x_j))\\
    & \times & (1+u(t_j)(y)-u(t_j)(x_j)),
  \end{eqnarray*}
  where the first equality relies on Lemma \ref{lem:gamma}.\par
  \noindent Recall that $\left\vert u(t_j)(y)-u(t_j)(x_j)\right\vert\le 1$, so that
  \begin{equation} \label{eq:oscill}
 0\le 1+u(t_j)(y)-u(t_j)(x_j)\le 2.
 \end{equation}
  If $\Delta u(t_j)(y)\ge \Delta u(t_j)(x_j)$, \eqref{eq:oscill} shows that
  \begin{equation} \label{eq:prod1}
  (\Delta u(t_j)(y)- \Delta u(t_j)(x_j))(1+u(t_j)(y)-u(t_j)(x_j))\ge 0.
  \end{equation}
  If $\Delta u(t_j)(y)< \Delta u(t_j)(x_j)<0$, then \eqref{eq:infdeltauj} implies 
  \begin{eqnarray*}
  \Delta u(t_j)(y)-\Delta u(t_j)(x_j) & \ge & \inf_V \Delta u(t_j)-\inf_{B_j} \Delta u(t_j)\\
  & \ge & -\frac{\varepsilon}2.
  \end{eqnarray*}
  This and \eqref{eq:oscill} imply
  \begin{equation} \label{eq:prod2}
  (\Delta u(t_j)(y)- \Delta u(t_j)(x_j))(1+u(t_j)(y)-u(t_j)(x_j))\ge -\varepsilon.
  \end{equation}
  Gathering \eqref{eq:prod1} and \eqref{eq:prod2} provides
  \[
  \Delta \Delta u(t_j)(x_j)+2 \Gamma(u(t_j), \Delta u(t_j))(x_j)\ge -\varepsilon.
  \]
  Hence,
  $$ 0 \geq \partial_{t} F(t)\vert_{t=t_j}(x_j) \geq t_j\left(-\varepsilon+\frac 2n(\Delta u(t_j)(x_j))^2\right)+\Delta u(t_j)(x_j).$$
  This inequality implies that
  \[
  \Delta u(t_j)(x_j)\ge -\frac{n}{4t_j}\left(1+\sqrt{1+\frac{8\varepsilon t_j}n}\right).
  \]
  We have therefore proved that, for all $\varepsilon>0$, there exists $j_1\ge 1$ such that, for all $j\ge j_1$,
  \[
  t_j\Delta u(t_j)(x_j)\ge -\frac n4\left(1+\sqrt{1+\frac{8\varepsilon t_j^2}n}\right),
  \]
  which, in turn, given the definition of $t_j$ and \eqref{eq:minbj}, yields
  \begin{equation} \label{eq:minsdelta}
  s\Delta u(s)(x)\ge -\frac n4\left(1+\sqrt{1+\frac{8\varepsilon t_j^2}n}\right)\ge -\frac n4\left(1+\sqrt{1+\frac{8\varepsilon T^2}n}\right)
  \end{equation}
  for all $s\in [0,T]$ and all $x\in B_j$. Since \eqref{eq:minsdelta} holds for all $j\ge j_1$, one obtains that, for all $x\in V$,
  \[
  s\Delta u(s)(x)\ge -\frac n4\left(1+\sqrt{1+\frac{8\varepsilon T^2}n}\right).
  \]
  Letting $\varepsilon$ go to $0$ therefore shows that
  \[
  s\Delta u(s)(x)\ge -\frac n2,
  \]
  and since this inequality holds for all $s\in [0,T]$ and all $T>0$, the proof of Theorem \ref{pro:liyau} is complete.
\end{proof}
\subsection{A Harnack inequality}
\noindent In order to prove the Harnack inequality (Theorem  \ref{pro:Harnack}), we need the following elementary result (\cite[Lemma 5.3]{MJMPA}):
\begin{lemma} \label{lem:elem}
  Let $T_{1} < T_{2} \in \R$. Assume that $\gamma :[T_{1},T_{2}] \rightarrow (0, \infty)$ is continuous  and that $C_{1}$, $C_{2}$ are positive constants. Then, the following inequality holds:
  $$ \frac{C_{2}^{2}}{C_{1}(T_{2}-T_{1})} \geq \inf_{s \in [T_{1},T_{2}]} \left( C_{2} \sqrt{\gamma (s)} -C_{1} \int_{s}^{T_{2}} \gamma (t) dt  \right). $$
  \end{lemma}
\noindent Let us now prove a Harnack inequality for solutions of the modified heat equation (Theorem \ref{pro:Harnack})

\begin{proof}
  We first assume that $y\sim x$. Recall that by the proof of Proposition \ref{prop:saut}, 
  $$ |u_{t}(x)-u_{t}(y)| \leq \sqrt{\frac{2\Gamma u_{t}(y)}{\alpha}}.$$
It follows that, for all $s\in (T_1,T_2)$,
  \begin{eqnarray*}
    u_{T_{1}}(x)-u_{T_{2}}(y) &=& - \int_{T_{1}}^{s} \partial_{t} u_{t}(x)dt +(u_{s}(x)-u_{s}(y)) -\int_{s}^{T_{2}} \partial_{t} u_{t}(y) dt\\
                              &\leq & \int_{T_{1}}^{T_{2}} \frac{n}{2t} dt - \int_{s}^{T_{2}} \Gamma u_{t} (y) dt + (u_{s}(x)-u_{s}(y))\\
    &\leq &\frac{n}{2} \log \left(  \frac{T_{2}}{T_{1}} \right) - \int_{s}^{T_{2}} \Gamma u_{t} (y) dt + \sqrt{\frac{2\Gamma u_{s}(y)}{\alpha}}
  \end{eqnarray*}
  where the second line follows from Theorem  \ref{pro:liyau}. To conclude, we minimize with respect to $s$ by using Lemma \ref{lem:elem} with 
  $C_{1}= \alpha/2$, 
$C_{2}=1$ and $\gamma(t)=\displaystyle \frac{2\Gamma u_{t}(y)}{\alpha}$. Thus, we get \\
 $$ u_{T_{1}}(x) -u_{T_{2}}(y)\leq\frac{n}{2} \log \left(  \frac{T_{2}}{T_{1}} \right) + \frac{2}{\alpha(T_{2}-T_{1})}.$$
  In the general case (that is, if  $x$ and $y$ are not neighbours), we write $T_{2}-T_{1}= N \delta$ where $N=d(x,y) \in \N$ (note that $N\ge 2$), so that $\delta = \displaystyle \frac{T_{2}-T_{1}}{d(x,y)}= \frac{T_{2}-T_{1}}{N}$. Set $t_i:=T_1+i\delta$ for all $i\in \llbracket 0,N\rrbracket$. Consider also $(x_{0},...,x_{N})$ a geodesic path joining $x$ and $y$. For any $i\in \llbracket 0,N-1\rrbracket$, since $x_{i}\sim x_{i+1}$, we have
  $$u_{t_{i}}(x_{i})-u_{t_{i+1}}(x_{i+1})  \leq \frac{n}{2}( \log(t_{i+1}) - \log(t_{i})) + \frac{2}{\alpha \delta}.$$
  By summing, we get
  $$ u_{t_{0}}(x_{0}) - u_{t_{N}} (x_{N})\leq \frac{n}{2} \log \left(  \frac{T_{2}}{T_{1}} \right)+ \frac{2N}{\alpha \delta}.$$
  Hence,
   $$ u_{T_{1}}(x) -u_{T_{2}}(y)\leq\frac{n}{2} \log \left(  \frac{T_{2}}{T_{1}} \right) + \frac{2d(x,y)^{2}}{\alpha(T_{2}-T_{1})}.$$
  \end{proof}
\section{The doubling volume property} \label{sec:doubling}
\noindent The main goal of this section is to prove the doubling volume property under $CD(0,n)$ (Theorem \ref{th:main}).

\noindent The proof is divided into two cases : large radii and small radii. The constant $C_{DV}$ in the previous theorem is just the maximum of the two $C_{DV}$ constants given in the next two propositions below.\par
\noindent Let us start with the case of large radii, following the strategy of \cite{M}.
\begin{proposition} \label{pro:doubling1}
Let $G$ be a weighted graph with bounded geometry. 
 Assume that \eqref{eq:CDKn} holds with $K=0$ and $n\in \N^{\ast}$. Then there exists $R_{min} >0$ and $C_{DV}>0$ depending on $n$ and on $\alpha$ such that, for all $x\in V$ and all $r \geq R_{\min}$,
\[
\mu(B(x,2r))\le C_{DV}\mu(B(x,r)).
\]
\end{proposition}
\begin{proof}
Let $\theta:=\gamma_0$ given by Proposition \ref{pro:semigroupcompar} and define
\begin{equation} \label{eq:defC}
C:=\gamma\frac{n}{\theta n+1} \sqrt{\alpha/2}
\end{equation}
where $\gamma>\frac 14$ will be chosen later. \par
\noindent Let $x\in V$ and $r > R_{\min}$ where $R_{\min} >0$ is chosen so that $R_{\min}> C/\sqrt{\alpha}$. Define, for all $y\in V$,
\[
u_0(y):=\max\left(-\frac Crd(x,y),-C\right).
\]
Note that $u_0\in L^\infty(V,\mu)$ and
\begin{equation} \label{eq:sizegradu0}
\left\Vert \Gamma u_0\right\Vert_\infty \leq \frac{ C^2}{2r^2} 
\leq \frac{ C^2}{2R_{\min}^2} < \alpha/2.
\end{equation}
Let $u_t$ be the solution of \eqref{eq:heatmodif} on $[0,\infty)$ such that $u(0)=u_0$ (recall that this solution is given by Theorem \ref{thm:solmodif}). Proposition \ref{pro:semigroupcompar} yields
\begin{equation} \label{eq:minorexp}
e^{\theta  u_{t}}  \ge  P_te^{\theta u_0}\ge P_t({\bf 1}_{V}+\theta u_0)={\bf 1}_{V}+\theta P_tu_0.
\end{equation}
The last inequality follows from the stochastic completeness of $G$ provided by \eqref{eq:completestoch}. \par
\noindent In a different connection, one has
\begin{eqnarray*}
(\Delta P_su_0)^2 & \le &  \frac{n}{2s}(P_s\Gamma u_0-\Gamma P_s u_0)\\
& \le & \frac {n}{s}  \left\Vert \Gamma u_0\right\Vert_\infty\\
& \le & \frac ns \frac{C^2}{2r^2}.
\end{eqnarray*}
where the first line is due to  \eqref{eq:gammapt2bis}, the second one relies on \eqref{eq:gradpt0} and the last one on \eqref{eq:sizegradu0}. As a consequence,
\begin{eqnarray*}
\left\vert P_tu_0-u_0\right\vert & \le & \int_0^t \left\vert \Delta P_su_0\right\vert ds\\
& \le & \sqrt{2}\frac{C\sqrt{nt}}r,
\end{eqnarray*}
which entails in particular that
\[
P_tu_0(x)\ge -\sqrt{2}\frac{C\sqrt{nt}}r
\]
since $u_0(x)=0$. This and \eqref{eq:minorexp} yield
\[
u_t(x)\ge \frac 1{\theta}\log\left(1-\theta\sqrt{2}\frac{C\sqrt{nt}}r\right).
\]
Let $R:=2r$, $y\in V$ such that $d(y,x)\le R$ and $T:=\displaystyle \frac{8R^{2}}{n \alpha}$. For all $t<T$, Theorem \ref{pro:Harnack} implies
\[
u_T(y)-u_t(x)\ge -\frac n2\log\left(\frac Tt\right)-\frac{2R^2}{\alpha(T-t)}.
\]
Hence, 
$$ u_{T}(y) \geq \left(  u_{t}(x)+ \frac{n}{2} \log(t)\right)- \left(\frac{n}{2}\log T + \frac{2R^{2}}{\alpha(T-t)} \right).$$
Set
\[
\sqrt{t}=\frac{r}{\theta C \sqrt{2n}}\frac{\theta n}{\theta n+1}.
\]
We compute
\begin{eqnarray*}
\frac 1{\theta}\log\left(1-\theta\sqrt{2}\frac{C\sqrt{nt}}r \right)+\frac n2\log t &=& - \frac{1}{\theta } \log(\theta n +1) + 
n \log \left( \frac{r}{\theta C \sqrt{2n }}  \right) \\
&- &n \log \left( 1+ \frac{1}{\theta n }\right)\\
&\geq & -n + n\log \left(  \frac{r}{\theta C \sqrt{2n}}\right)- \frac{1}{\theta}. \\
\end{eqnarray*}
Recall that $T:=\displaystyle \frac{8R^{2}}{n \alpha}$ and the definition of $C>0$ implies $t \leq T/2$. Indeed, this condition is equivalent to
\[
\frac{1}{C} \frac{n}{\theta n +1} \leq \frac{4\sqrt{2}}{\sqrt{\alpha}},
\]
which, in turn, is equivalent to
\[
C\ge \frac{n}{4(\theta n +1)} \sqrt{\alpha/2},
\]
which holds true by the definition of $C$ (see \eqref{eq:defC}) and the fact that $\gamma>\frac 14$. Thus,
\[
- \frac{n}{2} \log(T) - \frac{2R^{2}}{\alpha (T-t)} \geq -n \log \left(   \frac{2R}{\sqrt{n \alpha/2}} \right) -n.
\]
Thus, we get by combining all the previous estimates 
\begin{eqnarray*}
-u_{T}(y) & \leq & 2n + \frac{1}{\theta} + n \log \left(   \frac{4 \sqrt{2} \theta C}{\sqrt{\alpha/2}}\right)\\
& = & 2n + \frac{1}{\theta} + n\log \left( 4\sqrt{2} \gamma\frac{\theta n}{\theta n+1} \right) \\
&=: & Q.
\end{eqnarray*}
As a consequence, for all $\delta>0$,
\begin{equation} \label{eq:minornorml1}
\left\Vert e^{\delta u_T}\right\Vert_{L^1(B(x,R))}\ge e^{-\delta Q}\mu(B(x,R)).
\end{equation}
On the other hand, if $\delta >\gamma_1$ is fixed, Proposition \ref{pro:semigroupcompar} yields
\[
e^{\delta u_T}\le P_Te^{\delta u_0},
\]
which implies that
\begin{eqnarray*}
\left\Vert e^{\delta u_T}\right\Vert_{L^1(B(x,2r))} & \le & \left\Vert P_Te^{\delta u_0}\right\Vert_{L^1(B(x,2r))}.
\end{eqnarray*}
Notice that
\[
e^{\delta u_0}=e^{-C\delta}{\bf 1}_V+\left(e^{-C\delta\frac{d(x,\cdot)}r}-e^{-C\delta}\right){\bf 1}_{B(x,r)},
\]
from which one derives 
\[
P_T(e^{\delta u_0})=e^{-C\delta}{\bf 1}_V+P_T\left(\left(e^{-C\delta\frac{d(x,\cdot)}r}-e^{-C\delta}\right){\bf 1}_{B(x,r)}\right).
\]
As a consequence,
\begin{eqnarray}
\left\Vert e^{\delta u_T}\right\Vert_{L^1(B(x,R))} & \le & \left\Vert P_T(e^{\delta u_0})\right\Vert_{L^1(B(x,R))} \nonumber \\
& \le & e^{-C\delta}\mu(B(x,R))+\left\Vert P_T\left(\left(e^{-C\delta\frac{d(x,\cdot)}r}-e^{-C\delta}\right){\bf 1}_{B(x,r)}\right)\right\Vert_{L^1(V)} \nonumber\\
& \le & e^{-C\delta}\mu(B(x,R))+\left\Vert \left(e^{-C\delta\frac{d(x,\cdot)}r}-e^{-C\delta}\right){\bf 1}_{B(x,r)}\right\Vert_{L^1(V)}\nonumber\\
& \le & e^{-C\delta}\mu(B(x,R))+\mu(B(x,r)). \label{eq:marjornorml1}
\end{eqnarray}
Gathering \eqref{eq:minornorml1} and \eqref{eq:marjornorml1} yields
\[
e^{-\delta Q}\mu(B(x,R))\le e^{-C\delta}\mu(B(x,R))+\mu(B(x,r)),
\]
which means that
\[
\left(e^{-\delta Q}-e^{-\delta C}\right)\mu(B(x,R))\le \mu(B(x,r)).
\]
Notice that
\begin{eqnarray*}
Q-C & = & 2n + \frac{1}{\theta} + n\log \left( 4\gamma \sqrt{2}\frac{\theta n}{\theta n+1} \right)-\gamma\frac{n}{\theta n +1} 
\sqrt{\alpha/2} <-1
\end{eqnarray*}
provided that $\gamma>\frac 14$ is chosen large enough, only depending on $n$ and $\alpha$. It follows that
\begin{eqnarray*}
\mu(B(x,r)) & \ge & \mu(B(x,R))e^{-\delta Q}\left(1-e^{\delta(Q-C)}\right)\\
& \ge & \mu(B(x,R))e^{-\delta Q}\left(1-e^{-\delta}\right),
\end{eqnarray*}
which concludes the proof.
\end{proof}
\noindent We now consider the case of small radii.
\begin{proposition} \label{pro:doubling2}
Let $G$ be a weighted graph with bounded geometry. 
Assume that \eqref{eq:CDKn} holds with $K=0$ and $n\in \N^{\ast}$. Then there exists $C_{DV}>0$ depending on $n$ and on $\alpha$ such that, for all $x\in V$ and all $r \in [0,R_{\min}]$,
\[
\mu(B(x,2r))\le C_{DV}\mu(B(x,r)).
\]
\end{proposition}
\noindent Proposition \ref{pro:doubling2} is a straightforward consequence of Corollary \ref{pro:DVloc}. Note that here, the condition \eqref{eq:CDKn} is only used in the definition of $R_{\min}$.

\section*{Acknowledgement}
The authors would like to thank Rapha\"el Rossignol for useful discussions.

\bibliographystyle{plain}               
\bibliography{bibliocurgraph.bib} 
\end{document}